\documentclass[12pt]{article}

\usepackage{amssymb,amsmath}
\usepackage[all]{xy}
\usepackage[english,russian]{babel}
\usepackage[utf8x]{inputenc}
\usepackage{xcolor}
\usepackage[colorlinks,linkcolor=blue,citecolor=blue,filecolor=blue,urlcolor=blue]{hyperref}
\usepackage{graphicx}
\graphicspath{{pictures/}}

\newtheorem{theorem}{Theorem}[section]
\newtheorem{proposition}[theorem]{Proposition}
\newtheorem{lemma}[theorem]{Lemma}

\newtheorem{corollary}[theorem]{Corollary}

\newtheorem{remark}[theorem]{Remark}

\newtheorem{conjecture}[theorem]{{\bf Conjecture}}

\begin{document}
\leftline{UDC 512.7}
\begin{center}
	\Large \textbf{Rank Two Sheaves With Low Discriminant on the Fano Threefold of Index 2 and Degree~5}
\end{center}
\begin{center}
	\Large \textbf{D.~A.~Vassiliev\footnote{The study has been funded within the framework of the HSE University Basic Research Program.}}
\end{center}

\begin{abstract}
In this paper rank 2 Gieseker semistable sheaves $E$ on the Fano threefold $X_5$ of index 2 and degree 5 with maximal third Chern class $c_3(E)$ for all possible low values of discriminant $\overline{\Delta}_H(E)\le 40$ are described. The work uses the theory of tilt-stability and Bridgeland stability conditions on smooth projective threefolds. Also a conjecture about the rank 2 sheaves on $X_5$ with maximal $c_3$ and discriminant big enough is proposed.	
\end{abstract}


\section{Introduction}

The first general results on boundedness of the Chern class $c_3$ of rank 2 sheaves on a Fano threefold were obtained by R. Hartshorne in the case of $\mathbb P^3$ \cite{H}. In this work he obtained exact upper bounds on $c_3$ of stable reflexive rank 2 sheaves on $\mathbb P^3$ \cite[Theorem 8.2 (b), (d)]{H}. In 2018 B. Schmidt proved that the bounds of Hartshorne hold true for any Gieseker semistable rank 2 sheaves without an assumption of reflexivity and described semistable sheaves with maximal $c_3$ \cite{Sch18}. His work used the theory of tilt-stability and Bridgeland stability in the bounded derived category $\mathrm D^b(\mathbb P^3)$. The sheaves $E$ with maximal $c_3$, except for several special cases of low discriminant, were described as extensions
\begin{equation}\label{Schmidt 1}0\to\mathcal O_{\mathbb P^3}(-1)^{\oplus 2}\to E\to\mathcal O_V(m)\to 0\end{equation}
for $c_1(E)=-1$, respectively,
\begin{equation}\label{Schmidt 2}0\to F\to E\to\mathcal O_V(m)\to 0\end{equation}
for $c_1(E)=0$ (\cite[Theorem 3.1]{Sch18}). Here $V\subset\mathbb P^3$ is a projective plane and $F$ is a stable sheaf with $c_1(F)=-1$, the lowest possible $c_2(F)=1$ and maximal $c_3(F)=1$. In a later article \cite{Sch23} Schmidt generalized these results to sheaves of rank at most 4.

In our joint article with A. S. Tikhomirov \cite{Fano} we applied the methods of Schmidt to the smooth quadric threefold $X_2$. In this case we similarly were able to get a complete description of rank 2 (Gieseker) semistable sheaves on $X_2$ with maximal third Chern class, except for a single case $(c_1,c_2)=(0,1)$ in which these sheaves are not reflexive \cite[Theorem 3.1]{Fano}. In our joint article we also constructed several new infinite series of moduli components of rank 2 stable sheaves on $\mathbb P^3,X_2,X_5$ and on Fano threefolds $X_4$, which are intersections of two general\footnote{Here and
below by a general element of a family parametrized by a base $B$ we mean an element corresponding to a point in a dense open subset of $B$.} quadric hypersurfaces in $\mathbb P^5$ \cite[Theorem 1.3]{Fano}. The general sheaves in these components are constructed as extensions of the form $0\to F_i\to E\to G_j\to 0,$ where $G_j$ are certain sheaves supported on divisors $S\subset X$ (here $X$ is one of the threefolds $X_1,X_2,X_4,X_5$) and $F_i$ is either the sheaf $\mathcal O_X(-1)^{\oplus 2}$ or a twisted rank 2 sheaf $F$ generalizing the sheaves $F$ from (\ref{Schmidt 2}). We also obtained bounds on $c_3$ of stable reflexive rank two sheaves $E$ on $X_4$ and $X_5$ with $c_1(E)=0$ under an additional assumption of being a so-called sheaf of general type \cite[Theorem 1.2]{Fano}. These bounds were obtained by considering behaviour of stable sheaves under birational transformations $X_4\dashrightarrow X_1$,
$X_5\dashrightarrow X_2$ and they may be not sharp.

An important feature of $\mathbb P^3$ and the quadric $X_2$, used in \cite{Sch18,Fano}, is that bounded derived categories of coherent sheaves on these varieties admit full strong exceptional collections. It is known that there are only four types of smooth Fano threefolds admitting a full exceptional collection of (the minimal possible) length 4: these are $\mathbb P^3,X_2,X_5$ and Fano threefolds $V_{22}$ of index 1 and genus 12 \cite[Lemma 3.5]{NVdB}, \cite{Faenzi}. In the present article we concentrate on the variety $X_5$, which from now on will be denoted by $X$. 

The variety $X=X_5$ is the unique smooth Fano threefold of index 2 and degree 5 \cite{Исковских}. It has $\mathrm{Pic}\ X=\mathbb Z$. There are some articles in which vector bundles on $X$ were studied. Arrondo and Costa \cite{AC} classified rank 2 vector bundles on $X$ without intermediate cohomology, that is, rank 2 vector bundles $E$ such that $H^i(E(t))=0$ for all $0<i<3=\dim X$ and any $t\in\mathbb Z$. Bundles with such a vanishing are also called arithmetically Cohen--Macaulay, or aCM for short. Faenzi \cite{Faenzi} studied moduli spaces of semistable aCM bundles on $X$, which turn out to be unirational, and classified rank 3 semistable aCM bundles. Kuznetsov \cite{K} and Faenzi \cite{Even and odd} studied instanton bundles on $X$, while Lee and Park \cite{LP} described stable Ulrich bundles on this variety.

We will use the description of $X$ as a linear section of Grassmannian $\mathrm{Gr}(2,5)\subset\mathbb P^9$ (embedded by Pl\"ucker) by a general linear subspace of codimension 3. Denote by $\mathcal U$ the restriction to $X$ of the universal rank 2 subbundle on $\mathrm{Gr}(2,5)$, resp. by $\mathcal Q$ the restriction to $X$ of the universal rank 3 quotient bundle on $\mathrm{Gr}(2,5)$. The class of a divisor corresponding to the ample generator of the Picard group of $X$ will be denoted by $H$. Recall that the discriminant of a coherent sheaf $E$ on a variety $X$ with respect to a polarization $H$ is defined by
$$\overline{\Delta}_H(E)=(H^2\cdot\mathrm{ch}_1(E))^2-2(H^3\cdot\mathrm{ch}_0(E))(H\cdot\mathrm{ch}_2(E)),$$
where $\mathrm{ch}_i$ are the components of the Chern character. The Bogomolov--Gieseker inequality implies that the discriminant of a semistable sheaf is non-negative, and classification of sheaves with maximal $c_3$ in \cite{Sch18}, \cite{Fano} goes by induction on the discriminant. In the case of $X_5$ the cohomology groups $H^{2i}(X,\mathbb Z)$ are isomorphic to $\mathbb Z$ and the Chern classes of objects can be considered as integers. More precisely, for an object $E\in\mathrm D^b(X)$ with Chern classes $c_1,c_2,c_3$ we have
\begin{equation}\label{ch}\mathrm{ch}(E)=\mathrm{rk}(E)+c_1H+\left(\frac{c_1^2}{2}-\frac{c_2}{5}\right)H^2+\left(\frac{c_1^3}{6}-\frac{c_1c_2}{10}+\frac{c_3}{10}\right)H^3.\end{equation}
In particular, for a rank 2 sheaf $E$ we have $\overline{\Delta}_H(E)=20c_2(E)-25c_1(E)^2$. Using the theory of tilt-stability and Bridgeland stability, we were able to obtain the following description of semistable rank 2 sheaves with $\overline{\Delta}_H(E)\le 40$ and maximal $c_3$.

\begin{theorem}\label{intro}Let $E$ be a Gieseker semistable sheaf of rank 2 on $X$ with Chern classes $c_1,c_2,c_3$.\\
	(1) If $c_1=-1$, then $c_2\ge 2$. \\
	(1.1) If $c_2=2$, then $c_3\le 0$. In case of equality we have $E\cong\mathcal U$.\\
	(1.2) If $c_2=3$, then $c_3\le 1$. In case of equality $E$ is included into an exact triple
	$$0\to \mathcal U(-1)\to\mathcal O_X(-1)^{\oplus 4}\to E\to 0.$$  
	(2) If $c_1=0$, then $c_2\ge 0$.\\
	(2.1) If $c_2=0$, then $c_3\le 0$. In case of equality $E\cong\mathcal O_X^{\oplus 2}$.\\
	(2.2) If $c_2=1$, then $c_3\le -2$, and in case of equality $E$ is included into an exact triple
	$$0\to E\to\mathcal O_X^{\oplus 2}\to\mathcal O_L(1)\to 0$$
	for a projective line $L\subset X$. \\
	(2.3) If $c_2=2$, then $c_3\le 0$. In case of equality $E$ is included into an exact triple
	$$0\to \mathcal Q(-1)^{\oplus 2}\to\mathcal U^{\oplus 4}\to E\to 0.$$
\end{theorem} 

In fact, in Theorem \ref{main} we obtained the given bounds on the Chern classes of tilt-semistable objects, but Proposition \ref{2-stability} implies that Gieseker semistable sheaves are tilt-semistable everywhere above the largest wall. Theorem \ref{main} is formulated in terms of Chern characters, which are related to Chern classes by formula (\ref{ch}). Comparing the above result with the earlier works, we see that the bundle $\mathcal U$ from (1.1) is aCM. The sheaves from (2.3), when they are locally free, are at the same time aCM, instanton in sense of \cite{K}, and twisted by $\mathcal O_X(-1)$ Ulrich bundles.

Theorem \ref{intro} can be considered as a step towards the description of all Gieseker semistable rank 2 sheaves on $X_5$ with a maximal class $c_3$, by induction on the discriminant. Based on the cases of $\mathbb P^3$ and $X_2$, we can make a conjecture that general such sheaves $E$ for $c_2(E)\gg 0$ can be described as extensions $0\to F\to E\to G\to 0$, where $F$ is isomorphic to $\mathcal O_X(-1)^{\oplus 2}$ or $\mathcal U$, and $G$ is an invertible sheaf on a smooth hyperplane section of $X$ (see Conjecture \ref{conj2} for a precise statement). The author also hopes that methods of \cite{Sch18}, \cite{Fano} and the present article can help to describe semistable sheaves with maximal $c_3$ on some other Fano threefolds, using semiorthogonal decompositions of their derived categories.

Let us outline the structure of the paper. In Section \ref{derived category} we, firstly, recall the notions and results from the theory of tilt-stability and Bridgeland stability on threefolds that we will use. One of the main features of these stability conditions is that they are defined not on the category of coherent sheaves $\mathrm{Coh}(X)$, but on other hearts of bounded t-structures on $\mathrm D^b(X)$. These hearts can be obtained from $\mathrm{Coh}(X)$ using the operation of \textit{tilting}. Next, we consider the full strong exceptional collection $(\mathcal O_X(-1),\mathcal Q(-1),\mathcal U,\mathcal O_X)$ and the heart 
$$\mathfrak C=\langle\mathcal O_{X}(-1)[3],\mathcal Q(-1)[2],\mathcal U[1],\mathcal O_{X}\rangle$$ 
of a bounded t-structure on $\mathrm D^b(X)$ obtained from it. In (\ref{D}) we define a region $D$ of the upper half-plane with coordinates $(\beta,\alpha)$. Using this region we prove the existence of a family of Bridgeland stability conditions, defined on hearts $\mathcal A^{\alpha,\beta}(X)$, which are related to the heart $\mathfrak C$ by the procedure of tilting (see Corollary \ref{heart}). In terms of \cite{VP}, the exceptional collection generating $\mathfrak C$ satisfies the \textit{upper half-plane condition} for these stability conditions. This property can be used to describe tilt-semistable or Bridgeland semistable objects, because objects of $\mathfrak C$ with a given Chern character can be described rather explicitly (see Remark \ref{complexes}).

In Section \ref{bounds} we, firstly, give a characterization of the shifts $\mathcal O_X(n)^{\oplus m},\mathcal O_X(n)^{\oplus m}[1]$ of twisted trivial sheaves on $X$ in terms of stability and Chern characters and prove that such shifts have maximal $\mathrm{ch}_3$ among objects with the same $\mathrm{ch}_{\le 2}$. Next, we prove Theorem \ref{main}, giving Theorem \ref{intro} above.

Finally, in Section \ref{conjectures}, we make a conjecture (Conjecture \ref{conj2}) about sheaves with maximal $c_3$ for $c_2\gg 0$ and prove some bounds on the third Chern character which may help to prove this conjecture, following the arguments in \cite{Sch18}. Also, in Proposition \ref{O_S(D)} we determine the invertible sheaves on smooth hyperplane sections of $X$ which correspond to rank 2 semistable sheaves on $X$ with maximal $c_3$ if Conjecture \ref{conj2} is true.

\textbf{Acknowledgements.} The author thanks God for the opportunity to make this work. Also he thanks Alexander S. Tikhomirov for introducing him to the subject and encouraging the study.

\newpage
\textbf{Notation.}
\begin{center}
	\begin{tabular}{ r l }
		$\mathbb C$ & base field\ \\
		$X=X_5$ & smooth section of the Grassmannian $\mathrm{Gr}(2,5)$ embedded by\\
		& Pl\"ucker into the space $\mathbb P^9$, by a linear subspace $\mathbb P^6$\\
		$\mathcal U$ & restriction to $X$ of tautological subbundle on $\mathrm{Gr}(2,5)$\\
		$\mathcal Q$ & restriction to $X$ of tautological quotient bundle on $\mathrm{Gr}(2,5)$\\
		$H$ & positive generator of the Picard group $\mathrm{Pic}\ X\simeq\mathbb Z$ --\\ &
		class of hyperplane section of $X\hookrightarrow
		\mathbb P^6$,\\
		$\mathrm{Coh}(X)$ & category of coherent sheaves on $X$\\
		$\mathrm D^b(X)$ & bounded derived category of coherent sheaves on
		$X$\\
		$\mathcal H^{i}(E)$ & $i$th cohomology sheaf of the complex $E\in\mathrm
		D^b(X)$ \\
		$H^{i}(E)$ & $i$th hypercohomology group $\mathbf R^i\Gamma(E)$ of $E\in\mathrm
		D^b(X)$\\
		$\mathrm{ch}(E)$ & Chern character of the object $E \in\mathrm D^b(X)$ \\
		$\mathrm{ch}_{\le m}(E)$ & $\mathrm{ch}_0(E)+ \ldots+ \mathrm{ch}_m(E)$
	\end{tabular}
\end{center}

\section{The Derived Category $\mathrm D^b(X)$ and Stability}
\label{derived category}
Let $X=X_5$ be the unique smooth Fano threefold of Picard number one, index 2 and degree 5. Recall that $X$ can be constructed as a codimension 3 linear section of Grassmannian $\mathrm{Gr}(2,5)$ embedded by Pl\"ucker in $\mathbb P^9$. The cohomology groups of $X$ are 
$$H^*(X,\mathbb Z)=H^0(X,\mathbb Z)\oplus H^2(X,\mathbb Z)\oplus H^4(X,\mathbb Z)\oplus H^6(X,\mathbb Z)=$$
$$=\mathbb Z[X]\oplus\mathbb Z[H]\oplus\mathbb Z[L]\oplus\mathbb Z[P],$$
where $H$ denotes the class of a hyperplane section of $X$, $L$ is the class of a projective line on $X$ and $P$ is the class of a point. We have $H\cdot L=P,H^2=5L,H^3=5P$. We will identify $H^0(X,\mathbb Z)$ and $H^6(X,\mathbb Z)$ with $\mathbb Z$ by putting $X=1$ and $P=1$.

Recall that the \textit{slope} of a coherent sheaf $E\in\mathrm{Coh}(X)$ is defined as $\mu(E)=\frac{
H^2\cdot\mathrm{ch}_1(E)}{H^3\cdot\mathrm{ch}_0(E)}$ (in case of division by 0 we set $\mu(E)=
+\infty$). A coherent sheaf $E$ is called
$\mu$-\textit{stable} (resp. $\mu$\textit{-semistable}) or \textit{slope stable} (resp. \textit{slope semistable}) if for any proper
subsheaf $0\ne F\hookrightarrow E$ the inequality $\mu(F)<\mu(E/F)$ (resp. $\mu(F)\le\mu(E/F)$) holds.

Following \cite{BMT} and \cite{Sch18}, we obtain a new heart of a bounded t-structure on $\mathrm{D}^b(X)$ using the process of tilting. In general, let $\mathcal A$ be the heart of a bounded t-structure on a triangulated category $\mathfrak D$. Recall that a \textit{torsion pair} is a pair $(\mathcal T,\mathcal F)$ of subcategories of the category $\mathcal A$ such that, firstly, $\mathrm{Hom}(T,F)=0$ for any $T\in\mathcal T,F\in\mathcal F$ and, secondly, for any $E\in\mathcal A$ there is a short exact sequence $0\to T\to E\to F\to 0$ in $\mathcal A$ with $T\in\mathcal T,F\in\mathcal F$.

For an arbitrary real number $\beta$ we have the following torsion pair on the category of coherent sheaves:
$$\mathcal T_\beta=\{E\in\mathrm{Coh}(X)\colon\mathrm{any}\; \mathrm{quotient}\;E\to G\;\mathrm{satisfies}\;\mu(G)>\beta\},$$
$$\mathcal F_\beta=\{E\in\mathrm{Coh}(X)\colon\mathrm{any}\;\mathrm{subsheaf}\;F\to E\;\mathrm{satisfies}\;\mu(F)\le\beta\}.$$

In general, if $(\mathcal T,\mathcal F)$ is a torsion pair on the heart $\mathcal A$ as defined above, its \textit{tilt} is the extension closure $\langle\mathcal F[1],\mathcal T\rangle\subset\mathfrak D$. It is also the heart of a bounded t-structure.
In our case the new heart $\mathrm{Coh}^\beta(X)$ is the extension closure $\langle\mathcal F_\beta[1],\mathcal T_\beta\rangle$. Equivalently, $\mathrm{Coh}^\beta(X)$ can be described as the full additive subcategory of the derived category, consisting of those complexes $E\in\mathrm D^b(X)$ for which $\mathcal H^{-1}(E)\in\mathcal F_{\beta},\mathcal H^{0}(E)\in\mathcal T_{\beta}$ and $\mathcal H^{i}(E)=0$ for $i\neq -1,0$.

Recall the notion of tilt-stability. Again $\beta$ is an arbitrary real number. The twisted Chern character is defined as $\mathrm{ch}^\beta=e^{-\beta H}\cdot\mathrm{ch}$. Explicitly, for $E\in \mathrm D^b(X)$ with $\mathrm{ch}(E)=r+cH+dH^2+eH^3$ we have
$$\mathrm{ch}_0^\beta(E)=r,\; \mathrm{ch}_1^\beta(E)=(c-\beta r)H,\; \mathrm{ch}_2^\beta(E)=(d-\beta c+\frac{\beta^2}{2} r)H^2,$$
$$\mathrm{ch}_3^\beta(E)=(e-\beta d+\frac{\beta^2}{2}c-\frac{\beta^3}{6}r)H^3.$$

Let $\alpha>0$ be a positive real number. Following Schmidt (\cite{Sch14,Sch19}), for an object $E\in\mathrm{Coh}^\beta(X)$ we define a central charge function
$$Z^{\mathrm{tilt}}_{\alpha,\beta}(E)=Z^{\mathrm{tilt}}_{\alpha,\beta}(\mathrm{ch}_0(E),\mathrm{ch}_1(E),\mathrm{ch}_2(E))=-H\cdot\mathrm{ch}_2^\beta(E)+\frac{\alpha^2}{2}H^3\cdot\mathrm{ch}_0^\beta(E)+i(H^2\cdot\mathrm{ch}_1^\beta(E)),$$
which is a $\mathbb Z$-linear map $Z^{\mathrm{tilt}}_{\alpha,\beta}:K(\mathrm D^b(X))=K(\mathrm{Coh}^\beta(X))\to\mathbb C$.

The tilt-slope of $E$ is defined by
$$\nu_{\alpha,\beta}(E)=-\frac{\Re(Z^{\mathrm{tilt}}_{\alpha,\beta}(E))}{\Im(Z^{\mathrm{tilt}}_{\alpha,\beta}(E))},$$
where division by 0 is again interpreted as giving $+\infty$. The object $E$ is called \textit{tilt-(semi)stable} (or $\nu_{\alpha,\beta}$-\textit{(semi)stable}) if for any non-trivial proper subobject $F\hookrightarrow E$ in $\mathrm{Coh}^\beta(X)$ the inequality $\nu_{\alpha,\beta}(F)<(\le)\,\nu_{\alpha,\beta}(E/F)$ holds.

We can consider $\mathrm{ch}_{\le 2}$ as a map to a lattice $\Lambda=\mathbb Z^2\oplus\frac{1}{10}\mathbb Z$ (if $\mathrm{ch}_{\le 2}(E)=r+cH+dH^2$ we map $E$ to $(r,c,d)$). Define a bilinear form $Q^{\mathrm{tilt}}$ on $\Lambda\otimes\mathbb R$ by \linebreak $Q^{\mathrm{tilt}}((r,c,d),(R,C,D))=Cc-Rd-Dr$. Denote $Q^{\mathrm{tilt}}(E):=Q^{\mathrm{tilt}}(\mathrm{ch}_{\le 2}(E),\mathrm{ch}_{\le 2}(E))$.

The \textit{Bogomolov-Gieseker inequality} holds for tilt-stability, asserting non-negativity of the discriminant $\overline{\Delta}_H(E)=(H^3)^2\cdot Q^{\mathrm{tilt}}(E)$ (see \cite[Theorem 2.1]{MS18}):
\begin{proposition}\label{BG} For any $\nu_{\alpha,\beta}$-semistable object $E\in\mathrm{Coh}^\beta(X)$ we have
\begin{multline}\overline{\Delta}_H(E):=(H^3)^2\cdot Q^{\mathrm{tilt}}(E)=(H^2\cdot\mathrm{ch}_1(E))^2-2(H^3\cdot\mathrm{ch}_0(E))(H\cdot\mathrm{ch}_2(E))=\\=(H^2\cdot\mathrm{ch}_1^\beta(E))^2-2(H^3\cdot\mathrm{ch}_0^\beta(E))(H\cdot\mathrm{ch}_2^\beta(E))\ge 0.
\end{multline}
\end{proposition}

The $\nu_{\alpha,\beta}$-stability is a so-called \textit{very weak stability condition} \cite{Sch19}. In particular, it satisfies the \textit{support property} with respect to $Q^{\mathrm{tilt}}$. It means that, firstly, the Bogomolov-Gieseker inequality holds for semistable objects, and secondly, for any non-zero $v\in\Lambda\otimes\mathbb R$ with $Z^{\mathrm{tilt}}_{\alpha,\beta}(v)=0$ we have $Q^{\mathrm{tilt}}(v)<0$. Also Harder-Narasimhan filtrations exist for $\nu_{\alpha,\beta}$-stability. For a \textit{Bridgeland stability condition} $(\mathcal A,Z)$ defined on a heart $\mathcal A$ of a bounded t-structure, in addition, we have $Z(E)\neq 0$ for any nonzero $E\in\mathcal A$ \cite{Sch19}.

In order to construct Bridgeland stability conditions on $X$ we make another tilt, as proposed by Bayer, Macri and Toda for smooth projective threefolds (\cite{BMT}). Let
$$\mathcal T'_{\alpha,\beta}=\{E\in\mathrm{Coh}^\beta(X)\colon\text{any quotient}\, E\twoheadrightarrow G\;\mathrm{satisfies}\;\nu_{\alpha,\beta}(G)>0\},$$
$$\mathcal F'_{\alpha,\beta}=\{E\in\mathrm{Coh}^\beta(X)\colon\text{any non-trivial subobject}\, F\hookrightarrow E\;\mathrm{satisfies}\;\nu_{\alpha,\beta}(F)\le 0\},$$
and set $\mathcal A^{\alpha,\beta}(X)=\langle \mathcal F'_{\alpha,\beta}[1],\mathcal T'_{\alpha,\beta}\rangle$. For any $s\in\mathbb R$ define
$$Z_{\alpha,\beta,s}:=-\mathrm{ch}_3^\beta+s\alpha^2H^2\cdot\mathrm{ch}_1^\beta+i (\alpha H\cdot\mathrm{ch}_2^\beta-\frac{\alpha^3H^3}{2}\cdot\mathrm{ch}_0^\beta),$$
$$\lambda_{\alpha,\beta,s}:=-\frac{\Re(Z_{\alpha,\beta,s})}{\Im(Z_{\alpha,\beta,s})},$$
division by 0 is again interpreted as giving $+\infty$.

An object $E \in \mathcal A^{\alpha, \beta}(X)$ is called \textit{$\lambda_{\alpha, \beta, s}$-(semi)stable}, if for any non-trivial subobject $F \hookrightarrow E$ we have $\lambda_{\alpha, \beta, s}(F) < (\le)\  \lambda_{\alpha, \beta, s}(E)$.	

It is known that for $s>\frac 16$ $(\mathcal A^{\alpha,\beta}(X),Z_{\alpha,\beta,s})$ is a Bridgeland stability condition \cite[Corollary 0.2]{Li}. This result follows from the fact that the following \textit{generalized Bogomolov-Gieseker inequality} holds for $X$ (\cite[Theorem 0.1]{Li}, \cite[Theorem 5.4]{BMS16}):
\begin{proposition}\label{BMT} Assume that $E$ is $\nu_{\alpha,\beta}$-semistable. Then
$$Q_{\alpha,\beta}(E)=\alpha^2\overline{\Delta}_H(E)+4(H\cdot\mathrm{ch}_2^\beta)^2-6(H^2\cdot\mathrm{ch}_1^\beta(E))\mathrm{ch}_3^\beta(E)\ge 0.$$ \end{proposition}

Fix $v\in\Lambda=\mathbb Z^2\oplus\frac{1}{10}\mathbb Z$. There is a locally finite wall-and-chamber structure in the upper half-plane $\mathbb H:=\{(\beta,\alpha)\in\mathbb R^2\ |\ \alpha>0\}$ such that for an object $E\in\mathrm{Coh}^\beta(X)$ with $\mathrm{ch}_{\le 2}(E)=v$ tilt-stability of $E$ does not change when $(\beta,\alpha)$ vary within a chamber (\cite[Proposition B.5]{BMS16}). A \textit{numerical wall} with respect to $v\in\Lambda$ is a non-trivial proper subset $W$ of $\mathbb H$ given by an equation of the form $\nu_{\alpha,\beta}(v)=\nu_{\alpha,\beta}(w)$ for another element $w\in\Lambda$. We denote this numerical wall by $W(v,w)$. A numerical wall $W$ is called an \textit{actual wall} (or simply a \textit{wall}) if the set of semistable objects with class $v$ changes at $W$. The structure of walls in tilt-stability is well understood. We have the following Proposition.

\begin{proposition}[Numerical properties of walls]\label{numerical} Let $v=(v_0,v_1,v_2)\in\Lambda$ be a fixed class with $\overline{\Delta}_H(v)\ge 0$. All numerical walls in the following statements are with respect to $v$.
	\begin{enumerate} 
		\item Numerical walls are either semicircles with center on the $\beta$-axis or rays parallel to the $\alpha$-axis. If $v_0\neq 0$, there is exactly one numerical vertical wall given by $\beta = v_1/v_0$. If $v_0 = 0$, there are no vertical walls.
		\item The curve $\nu_{\alpha, \beta}(v) = 0$ intersects all semicircular walls at their highest point.
		\item If $v_0\neq 0$, then the curve $\nu_{\alpha, \beta}(v) = 0$ is a hyperbola, which may be degenerate if $\overline{\Delta}_H(v)=0$. Its asymptotes are the lines $\beta-\alpha=v_1/v_0$ and $\beta+\alpha=v_1/v_0$. The semicircular numerical walls are nested along one of the two branches of the hyperbola. 
		\item If there exist semicircular actual walls, then there is a largest such wall.
		\item If $\overline{\Delta}_H(v)>0$, then the set of $(\beta,\alpha)$ for which $Q_{\alpha,\beta}(v_0,v_1H,v_2H^2,v_3H^3)=0$ is the semicircular numerical wall $W((v_0,v_1,v_2),(5v_1,10v_2,15v_3))$. The locus where $Q_{\alpha,\beta}(v_0,v_1H,v_2H^2,v_3H^3)<0$ is the semidisk bounded by $Q_{\alpha,\beta}(v_0,v_1H,v_2H^2,v_3H^3)=0$ and $\beta$-axis.
	\end{enumerate}
\end{proposition}

\textit{Proof.} The statements (1)-(3) were proven by Schmidt \cite[Theorem 3.3]{Sch19} for the case of $X=\mathbb P^3$. In that case $v,w\in\mathbb Z^2\oplus\frac 12\mathbb Z$, but the calculations can be done for any $v,w\in\mathbb R^3$. The statement (4) follows from \cite[Lemma 7.3]{Sch19}. The statement (5) for the case of $\mathbb P^3$ is \cite[Lemma 2.10]{Sch18}, and the calculation can be done in our case too (we only need to multiply the vector $(v_1,2v_2,3v_3)$ by 5 so that $5(v_1,2v_2,3v_3)\in\Lambda$). $\square$

If an object $E\in\mathrm{Coh}^\beta(X),$ $\mathrm{ch}_{\le 2}(E)=v$ is destabilized at the wall $W$, that is, if $W$ belongs to the boundary of the set 
$\{(\beta,\alpha)\in\mathbb H\ |\ E\ \mathrm{is}\ \nu_{\alpha,\beta}\mathrm{-semistable}\}$, 
then either $W$ is a vertical wall, or there exists a short exact sequence $0\to F\to E\to G\to 0$ in $\mathrm{Coh}^\beta(X)$ for $(\beta,\alpha)\in W$ such that $$W=\{(\beta,\alpha)\in\mathbb H\ |\ \nu_{\alpha,\beta}(E)=\nu_{\alpha,\beta}(F)\}$$ (\cite[Proposition B.5]{BMS16}). Following Schmidt, we denote a wall given by such an equation by $W(E,F)$. If $W(E,F)$ is semicircular, we denote its radius by $\rho(E,F)$. 
We will also use the properties of walls described in the following Proposition.

\begin{proposition}[Properties of actual walls]\label{actual}Let $0\to F\to E\to G\to 0$ be an exact sequence in $\mathrm{Coh}^\beta(X)$ defining a wall $W$.
	\begin{enumerate}
		\item If $W$ is semicircular and $\mathrm{ch}_0(F)>\mathrm{ch}_0(E)\ge 0$, then
		$$\rho(E,F)^2\le\frac{\overline{\Delta}_H(E)}{4H^3\cdot\mathrm{ch}_0(F)(H^3\cdot(\mathrm{ch}_0(F)-\mathrm{ch}_0(E))}.$$
		\item If $W$ is semicircular, then $\overline{\Delta}_H(F)+\overline{\Delta}_H(G)<\overline{\Delta}_H(E).$
		\item For any point $(\beta,\alpha)\in W$, we have $0\le\mathrm{ch}_1^\beta(F)\le\mathrm{ch}_1^\beta(E)$. In case of equality on either end, the wall is vertical. In particular, if $H^2\cdot\mathrm{ch}_1^{\beta_0}(E)>0$ has the minimal positive value among objects of $\mathrm{Coh}^{\beta_0}(X)$, then there is no actual wall that intersects the line $\beta=\beta_0$.
		\item If $\mathrm{ch}_0(E)>0$, then either $\mathrm{ch}_0(F)>0$ and $\mu(F)\le\mu(E)$ or $\mathrm{ch}_0(G)>0$ and $\mu(G)\le\mu(E)$. 
		\item Suppose that $W$ is semicircular. If $\mathrm{ch}_0(E)>0, \mathrm{ch}_0(F)>0$ and $\mu(F)\le\mu(E)$, then $\beta_-(E)<\beta_-(F)\le\mu(F)<\mu(E)$. 
	\end{enumerate}
\end{proposition}

\textit{Proof.}
	The statement (1) is proven in \cite[Lemma 2.4]{MS18}. The statement (2) is essentially proven in \cite[Lemma 2.7]{Sch19}: it suffices to note that in the case of a semicircular wall we have $\mathrm{ch}_{\le 2}(F)\neq 0,\mathrm{ch}_{\le 2}(G)\neq 0$ and $\mathrm{ch}_{\le 2}(F)-\lambda\mathrm{ch}_{\le 2}(G)\neq 0$ for any $\lambda\in\mathbb R$. Hence, if we apply this lemma with the quadratic form $Q(E)=\overline{\Delta}_H(E)$ defined on $\Lambda\otimes\mathbb R$, we get that for some $\lambda>0$ and $(\beta,\alpha)\in W$ $Z^{\mathrm{tilt}}_{\alpha,\beta}(\mathrm{ch}_{\le 2}(F)-\lambda\mathrm{ch}_{\le 2}(G))=0$, so
	$Q(\mathrm{ch}_{\le 2}(F)-\lambda\mathrm{ch}_{\le 2}(G))<0$, hence $Q(F,G):=Q^{\mathrm{tilt}}(\mathrm{ch}_{\le 2}(F),\mathrm{ch}_{\le 2}(G))>0$ (we use the Bogomolov-Gieseker inequality for $F$ and $G$) and $Q(E)=Q(F)+Q(G)+2Q(F,G)>Q(F)+Q(G)$.
	
	The statement (3) is proven in \cite[Proposition 2.7 (iii)]{Sch23} for the case of $\mathbb P^3$, and the proof works for any smooth projective threefold. The statement (4) is proven in \cite[Proposition 2.7 (v)]{Sch23}, the proof works in our case unchanged. 
	
	Finally, the statement (5) appeared in the proof of \cite[Lemma 3.6]{BMSZ} and was formulated in \cite[Proposition 2.7 (vi)]{Sch23}. $\square$

We have the following lemma (\cite[Lemma 3.5]{Sch19}) that refines the last statement of Proposition \ref{actual} (3).

\begin{lemma}\label{minimal} If an object $E\in\mathrm{Coh}^\beta(X)$ is $\nu_{\alpha,\beta}$-semistable for $\alpha\gg 0$, then one of the following holds:\\
(a) $\mathcal H^{-1}(E)=0$ and $\mathcal H^0(E)$ is a $\mu$-semistable pure sheaf of dimension $\ge 2$.\\
(b) $\mathcal H^{-1}(E)=0$ and $\dim \mathcal H^0(E)\le 1$.\\
(c) $\mathcal H^{-1}(E)$ is a torsion-free $\mu$-semistable sheaf and $\dim \mathcal H^0(E)\le 1$. Also, if $\nolinebreak{\beta>\mu(E)}$, then there are no non-zero morphisms from sheaves of dimension $\le 1$ to $E$.\\
Moreover, if $H^2\cdot\mathrm{ch}_1^\beta(E)$ is either zero or has the minimal positive value among objects of $\mathrm{Coh}^\beta(X)$, then $E$ is $\nu_{\alpha,\beta}$-semistable if and only if it satisfies one of the three properties above.
\end{lemma}

\textit{Proof.} This is proven in \cite[Lemma 7.2.1, Lemma 7.2.2]{BMT}. Note that, if we assume that $H^2\cdot\mathrm{ch}_1^\beta(E)$ has the minimal positive value and that $E$ is $\nu_{\alpha,\beta}$-semistable, then it is $\nu_{\alpha,\beta}$-stable. $\square$

\begin{remark}\label{line bundles} Proposition \ref{actual} (2) implies that objects $E\in\mathrm{Coh}^\beta(X)$ with $\overline{\Delta}_H(E)=0$ can only be destabilized at the vertical wall. Hence, by Lemma \ref{minimal} line bundles and their shifts by one are tilt-semistable everywhere (under the condition that they belong to $\mathrm{Coh}^\beta(X)$). \end{remark}

We will also use the following relation between tilt-stability and the more classical notions of stability. Let $f, g \in \mathbb R[m]$ be polynomials. If $\deg(f) < \deg(g)$, then we set
$f > g$. If $\deg(f) = \deg(g)$ and $a$, $b$ are the leading coefficients in $f$,
$g$ respectively, then we put $f < (\leq)\ g$ if $\frac{f(m)}{a} < (\leq)\
\frac{g(m)}{b}$ for all $m \gg 0$.
For an arbitrary sheaf $E\in\mathrm{Coh}(X)$ we define the numbers $a_i(E)$ for $i \in \{0,
1, 2, 3\}$ through the Hilbert polynomial $P(E, m):=\chi(E(m))=a_3(E)m^3+a_2(E) m^2+
a_1(E) m + a_0(E)$. In addition, we set $P_2(E,m):=a_3(E)m^2+a_2(E)m+a_1(E)$.
The sheaf $E \in \mathrm{Coh}(X)$ is called \emph{Gieseker (semi)stable}, if for any
its proper subsheaf $0\ne F \hookrightarrow E$ the inequality $P(F, m) < (\leq)\ P(E/F, m)$ holds. Respectively, $E$ is called \emph{(semi)stable}
if for any subsheaf $0\ne F \hookrightarrow E$ the inequality $P_2(F, m) < (\leq)\ P_2(E/F, m)$ holds.
Gieseker stability, 2-stability and $\mu$-stability of a sheaf satisfy the relations:
\begin{equation*}
	\xymatrix{
		\text{$\mu$-stability} \ar@{=>}[r] & \text{$2$-stability} \ar@{=>}[r] & \text{Gieseker stability} \ar@{=>}[d] \\
		\text{$\mu$-semistability} & \text{$2$-semistability} \ar@{=>}[l] &\text{Gieseker semistability} \ar@{=>}[l]
	}
\end{equation*}

\begin{proposition}\label{2-stability}
	An object $E\in\mathrm D^b(X)$ belongs to $\mathrm{Coh}^\beta(X)$ and is $\nu_{\alpha,\beta}$-(semi)stable for $\beta<\mu(E)<+\infty$ 
	and $\alpha\gg0$ iff $E$ is a 2-(semi)stable sheaf of positive rank.
\end{proposition}
\textit{Proof.}
	This is proven in \cite[Prop. 14.2]{Bri} for K3 surfaces, but the proof works in our case too. $\square$

For a tilt-semistable object $E$ with $\mathrm{ch}_0(E)\neq 0$, following \cite{BMSZ}, we define
$$\beta_-(E)=\mu(E)-\sqrt{\frac{\overline{\Delta}_H(E)}{(H^3\cdot\mathrm{ch}_0(E))^2}},$$
$$\beta_+(E)=\mu(E)+\sqrt{\frac{\overline{\Delta}_H(E)}{(H^3\cdot\mathrm{ch}_0(E))^2}}.$$

The numbers $\beta_-(E)$ and $\beta_+(E)$ are the two solutions of the equation $\nu_{0,\beta}(E)=0$. We will use the following theorem to bound the rank of destabilizing subobjects.

\begin{theorem}\label{rank bound}Let $E$ be a tilt-stable object on $X$ with $\mathrm{ch}_0(E)\neq 0$, which is not a shift of a line bundle or a twisted ideal sheaf of points. If there is an integer $n\in\mathbb Z$ such that $\beta_-(E),\beta_+(E)\in[n,n+1)$, then
	$$\overline{\Delta}_H(E)\ge\mathrm{ch}_0(E)^2.$$
\end{theorem}

\textit{Proof.} This is a special case of \cite[Theorem 3.1]{BMSZ} (applied to $E(-nH)$) with $H^3=5,i_X=2$. $\square$

In order to apply Theorem \ref{rank bound} we will use the following lemma.

\begin{lemma}\label{stable}Let $E$ be a tilt-semistable object with $\mathrm{ch}_0(E)>0$. If there is a semicircular wall $W$ for $E$, then $W$ is induced by a tilt-stable subobject or quotient $F$ with $\mathrm{ch}_0(F)>0,\mu(F)<\mu(E).$
\end{lemma}

\textit{Proof.} Choose a point $(\beta_0,\alpha_0)\in W$. The wall $W$ is given by a short exact sequence
\begin{equation}\label{F' G'}0\to F'\to E\to G'\to 0\end{equation}
in $\mathrm{Coh}^{\beta_0}(X)$. Suppose that (\ref{F' G'}) destabilizes $E$ below $W$. Choose $\alpha_1<\alpha_0$ so that the point $(\beta_0,\alpha_1)$ lies in the chamber adjacent to $W$. Since $E$ is not $\nu_{\alpha_1,\beta_0}$-semistable, we can take a tilt-stable subobject $F\hookrightarrow E$ such that $\nu_{\alpha_1,\beta_0}(F)>\nu_{\alpha_1,\beta_0}(E)$. We should have $\nu_{\alpha_0,\beta_0}(F)=\nu_{\alpha_0,\beta_0}(E)$ and $\nu_{\alpha,\beta_0}(F)>\nu_{\alpha,\beta_0}(E)$ for $\alpha<\alpha_0$, hence $\frac{\mathrm{ch}_0(F)}{\mathrm{ch}_1^{\beta_0}(F)}>\frac{\mathrm{ch}_0(E)}{\mathrm{ch}_1^{\beta_0}(E)}>0$. Therefore $\frac{\mathrm{ch}_1^{\beta_0}(F)}{\mathrm{ch}_0(F)}<\frac{\mathrm{ch}_1^{\beta_0}(E)}{\mathrm{ch}_0(E)}$ and we obtain $\mu(F)<\mu(E),\mathrm{ch}_0(F)>0$. Finally, if (\ref{F' G'}) destabilizes $E$ above $W$, then we analogously choose a point $(\beta_0,\alpha_1)$ in the chamber located immediately above $W$ and take a tilt-stable quotient $E\twoheadrightarrow F$ such that $\nu_{\alpha_1,\beta_0}(E)>\nu_{\alpha_1,\beta_0}(F)$. $\square$

In this article we will mostly work with tilt-stability. The following lemma will allow us to use the properties of Bridgeland stability.

\begin{lemma}\label{nu-lambda}
Let $E\in\mathrm{Coh}^{\beta_0}(X)$ be a $\nu_{\alpha_0,\beta_0}$-semistable object such that $\nu_{\alpha_0,\beta_0}(E)=0$ and fix some $s>0$. Then $E[1]$ is $\lambda_{{\alpha_0},{\beta_0},s}
$-semistable. If moreover $E$ is $\nu_{\alpha_0,\beta_0}$-stable, then there is a neighbourhood $U$ of $(\beta_0,\alpha_0)$ such that such that for all $(\beta,\alpha)\in U$ with $\nu_{\alpha,\beta}(E)
>0$ the object $E$ is $\lambda_{\alpha,\beta,s}$-semistable.
\end{lemma}

\textit{Proof.} All the statements are proved in \cite[Lemma 6.2]{Sch19}, except for the case of a strictly $\nu_{\alpha_0,\beta_0}$-semistable object $E$. However, such an object also belongs to the category $\mathcal A^{{\alpha_0}, {\beta_0}}(X)$ and has $\lambda_{{\alpha_0},{\beta_0},s}(E)=+\infty$, hence it is $\lambda_{{\alpha_0},{\beta_0},s}$-semistable. $\square$

Recall that $X=X_5$ is a codimension 3 linear section of $\mathrm{Gr}(2,V)$, where $V\cong\mathbb C^5$. Denote by $\mathcal U$ and $\mathcal Q$ the restrictions to $X$ of tautological subbundle and tautological quotient bundle on the Grassmannian, so that we have a short exact sequence
$$0\to\mathcal U\to V\otimes\mathcal O_{X}\to\mathcal Q\to 0.$$

In $\mathrm{D}^b(X)$ we have a full strong exceptional collection $(\mathcal O_X(-1),\mathcal Q(-1),\mathcal U,\mathcal O_X)$, it is the collection $\varepsilon_2$ from \cite{Orlov} up to a twist with a line bundle. More precisely, Orlov proved that it is a strong exceptional collection. Kuznetsov notes in \cite[(12)]{K} that $(\mathcal U,\mathcal Q^\vee,\mathcal O_X,\mathcal O_X(1))$ is a full exceptional collection in $\mathrm{D}^b(X)$, and the collection that we use can be obtained from it by mutations and twisting with line bundles (\cite[Lemma 5.3]{Faenzi}), hence it is also full. The Chern classes of $\mathcal U$ and $\mathcal Q$ are calculated in \cite{LP}. Using this we can obtain the following lemma.

\begin{lemma} For all $n\in\mathbb Z$ we have
$$\mathrm{ch}^\beta(\mathcal O_{X}(n))=1+(n-\beta)H+(\frac{n^2}{2}-n\beta+\frac{\beta^2}{2})H^2+(\frac{n^3}{6}-\frac{n^2}{2}\beta+n\frac{\beta^2}{2}-\frac{\beta^3}{6})H^3.$$

The twisted Chern characters of $\mathcal U$ and $\mathcal Q^\vee$ are given by
$$\mathrm{ch}^\beta(\mathcal Q(-1))=3+(-2-3\beta)H+(\frac 25+2\beta+\frac 32\beta^2)+(\frac{1}{15}-\frac 25\beta-\beta^2-\frac{\beta^3}{2})H^3,$$
$$\mathrm{ch}^\beta(\mathcal U)=2+(-1-2\beta)H+(\frac{1}{10}+\beta+\beta^2)H^2+(\frac{1}{30}-\frac{\beta}{10}-\frac{\beta^2}{2}-\frac{\beta^3}{3})H^3.$$

The values of $Z_{\alpha,\beta,\frac{1}{6}}$ can be computed as
$$Z_{\alpha,\beta,\frac 16}(\mathcal O_X(-1))=\frac 16+\frac{\beta}{2}+\frac{\beta^2}{2}+\frac{\beta^3}{6}-\frac{\alpha^2}{6}(1+\beta)+i\alpha(\frac 12+\beta+\frac{\beta^2}{2}-\frac{\alpha^2}{2}),$$
$$Z_{\alpha,\beta,\frac 16}(\mathcal Q(-1))=-\frac{1}{15}+\frac 25\beta+\beta^2+\frac{\beta^3}{2}-\frac{\alpha^2}{6}(2+3\beta)+i\alpha(\frac 25+2\beta+\frac 32\beta^2-\frac 32\alpha^2),$$
$$Z_{\alpha,\beta,\frac 16}(\mathcal U)=-\frac{1}{30}+\frac{\beta}{10}+\frac{\beta^2}{2}+\frac{\beta^3}{3}-\frac{\alpha^2}{6}(1+2\beta)+i\alpha(\frac{1}{10}+\beta+\beta^2-\alpha^2),$$
$$Z_{\alpha,\beta,\frac 16}(\mathcal O_{X})=\frac{\beta^3}{6}-\frac{\alpha^2\beta}{6}+i\alpha(\frac{\beta^2}{2}-\frac{\alpha^2}{2}).$$
\end{lemma}

Since $(\mathcal O_X(-1),\mathcal Q(-1),\mathcal U,\mathcal O_X)$ is a full strong exceptional collection, it follows that the extension closure $\mathfrak C:=\langle\mathcal O_{X}(-1)[3],\mathcal Q(-1)[2],\mathcal U[1],\mathcal O_{X}\rangle$ is the heart of a bounded t-structure on $\mathrm{D}^b(X)$ \cite[Lemma 3.14]{Curves}.

Define the following region in the upper half-plane:
\begin{equation}\label{D}D:=\{(\beta,\alpha)\in\mathbb H\ |\ \beta< -\frac 12,\alpha<\beta+1\}.\end{equation}
Following Schmidt (\cite[Lemma 4.4, Lemma 4.6]{Sch14}), we prove the next lemma. The part (i) of it asserts, roughly speaking, that the $Z_{\alpha,\beta,s}$-values of objects of $\mathfrak C$ lie in one half-plane.

\begin{lemma}\label{two properties}Take $(\beta,\alpha)\in D$.\\
(i) There exists $s>\frac 16$ and $\phi_1\in\mathbb R$ such that
$$Z_{\alpha,\beta,s}(\mathfrak C)\subset\{re^{\pi\phi i}\colon r\ge 0,\phi_1-1<\phi\le\phi_1\},$$ 
(ii) $\mathfrak C\subset\langle\mathcal A^{\alpha,\beta}(X),\mathcal A^{\alpha,\beta}(X)[1]\rangle$.
\end{lemma}

\textit{Proof.} (i) We can plot the sets of points $(\beta,\alpha)\in\mathbb R^2$ where two of complex numbers $Z_{\alpha,\beta,\frac 16}(\mathcal O_{X}(-1)),Z_{\alpha,\beta,\frac 16}(\mathcal Q(-1)),Z_{\alpha,\beta,\frac 16}(\mathcal U),Z_{\alpha,\beta,\frac 16}(\mathcal O_{X})$ are linearly dependent over $\mathbb R$.
\begin{figure}[h]
\center{\includegraphics[scale=0.4]{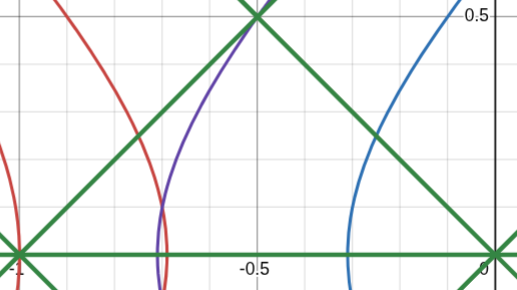}}
\caption{Curves in which some of $Z_{\alpha,\beta,\frac 16}(\mathcal O_{X}(-1)),Z_{\alpha,\beta,\frac 16}(\mathcal Q(-1)),Z_{\alpha,\beta,\frac 16}(\mathcal U),Z_{\alpha,\beta,\frac 16}(\mathcal O_{X})$ are $\mathbb R$-linearly dependent. The horizontal line is the $\beta$-axis, the vertical one is the $\alpha$-axis.}
\end{figure}
The graph shows that the region $D$ is divided into four subregions (by the purple curve, where $Z_{\alpha,\beta,\frac 16}(\mathcal O_X(-1)$ is proportional to $Z_{\alpha,\beta,\frac 16}(\mathcal Q(-1))$, and a red curve, where $Z_{\alpha,\beta,\frac 16}(\mathcal O_X)$ is proportional to $Z_{\alpha,\beta,\frac 16}(\mathcal U)$). It suffices to prove that for one point in each of the subregions the $Z_{\alpha,\beta,\frac 16}$-values of objects $\mathcal O_{X}(-1)[3],\mathcal Q(-1)[2],\mathcal U[1]$ and $\mathcal O_{X}$ belong to one half-plane as described in the statement. Indeed, for any $(\beta,\alpha)\in D$ the $Z_{\alpha,\beta,\frac 16}$-values of these four objects do not vanish (otherwise at least three curves would intersect $D$) and $Z_{\alpha,\beta,s}(E)$ depends continuously on $s$, hence there exists $s>\frac 16$ such that $Z_{\alpha,\beta,s}$-values of four generators of $\mathfrak C$ also lie in one half-plane. A calculation shows that at those points of $D$ where $5\alpha^2 - 5\beta^2 - 12\beta - 6 < 0$ (to the right of the purple curve on the plot) for some $s>\frac 16$ we can take $\phi_1=\frac 1\pi\arg Z_{\alpha,\beta,s}(\mathcal O_X(-1)[3])$, respectively, if $5\alpha^2 - 5\beta^2 - 12\beta - 6 >  0$, then we take $\phi_1=\frac 1\pi\arg Z_{\alpha,\beta,s}(\mathcal Q(-1)[2])$, and if $5\alpha^2 - 5\beta^2 - 12\beta - 6 = 0$, then at least one of these two values works. 

(ii) Note that $\nu_{\alpha,\beta}(\mathcal O_X)>0$ for $(\beta,\alpha)\in D$, so tilt-stability of line bundles implies that $\mathcal O_X\in\mathcal T'_{\alpha,\beta}\subset\mathcal A^{\alpha,\beta}(X)$. And $\nu_{\alpha,\beta}(\mathcal O_X(-1)[1])<0$ in $D$, so $\mathcal O_X(-1)[3]\in\mathcal F'_{\alpha,\beta}[2]\subset\mathcal A^{\alpha,\beta}(X)[1]$.

Let us show that $\mathcal Q(-1)[2]\in\mathcal A^{\alpha,\beta}(X)[1]$. Note that the vector bundle $\mathcal Q(-1)$ is $\mu$-stable by the criterion \cite[Remark 1.2.6 b)]{OSS} (which holds for $X_5$ too), because $h^0(\mathcal Q(-1))=0$ (since $(\mathcal Q(-1),\mathcal O_{X})$ is an exceptional pair) and $h^0((\mathcal Q(-1))^\vee(-1))=h^0(\mathcal Q^\vee)=0$ (since $(\mathcal Q^\vee,\mathcal O_{X})$ is an exceptional pair). Assume that $\beta<-\frac 23$. Note that $\mathcal Q(-1)$ is $\nu_{\alpha,-1}$-semistable by Lemma \ref{minimal}, because $H^2\cdot\mathrm{ch}_1^{-1}(\mathcal Q(-1))$ has the least positive value (among objects of $\mathrm{Coh}^{-1}(X)$). Since $\beta_-(\mathcal Q(-1))<-1$, $\mathcal Q(-1)$ is also $\nu_{\alpha,\beta}$-semistable for all $-1<\beta<-\frac 23,\alpha>0$. Since $\nu_{\alpha,\beta}(\mathcal Q(-1))<0$ for $-1<\beta<-\frac 23$, we get $\mathcal Q(-1)[2]\in\mathcal F'_{\alpha,\beta}[2]\subset\mathcal A^{\alpha,\beta}(X)[1]$. Assume now that $\beta\ge-\frac 23$. Note that $\mathcal Q(-1)[1]$ is $\nu_{\alpha,-\frac 12}$-semistable by Lemma \ref{minimal}, since $H^2\cdot\mathrm{ch}_1^{-\frac 12}(\mathcal Q(-1)[1])$ has the least positive value. Also $\beta_+(\mathcal Q(-1)[1])>-\frac 12$, hence $\mathcal Q(-1)[1]$ is $\nu_{\alpha,\beta}$-semistable for all $-\frac 23\le\beta<-\frac 12,\alpha>0$. We have $\nu_{\alpha,\beta}(\mathcal Q(-1)[1])>0$ for $-\frac 23\le\beta<-\frac 12$, so $\mathcal Q(-1)[2]\in\mathcal T'_{\alpha,\beta}[1]\subset\mathcal A^{\alpha,\beta}(X)[1]$.

Since $(\mathcal T'_{\alpha,\beta},\mathcal F'_{\alpha,\beta})$ is a torsion pair on the category $\mathrm{Coh}^\beta(X)$, for $\beta<-\frac 12$ we get an exact sequence
$$0\to \mathcal U'\to\mathcal U\to \mathcal U''\to 0,$$
where $\mathcal U'\in\mathcal T'_{\alpha,\beta},\mathcal U''\in\mathcal F'_{\alpha,\beta}$. Then $\mathcal U'[1]\in\mathcal A^{\alpha,\beta}(X)[1],\mathcal U''[1]\in \mathcal A^{\alpha,\beta}(X)$, so $\mathcal U[1]\in\nolinebreak\langle\mathcal A^{\alpha,\beta}(X),\mathcal A^{\alpha,\beta}(X)[1]\rangle$. $\square$

For the purpose of reference we include the following lemma.

\begin{lemma}\label{C A} Let $\mathcal C$ and $\mathcal A$ be two hearts of bounded t-structures on a triangulated category $\mathfrak D$. If $\mathcal C\subset\langle\mathcal A,\mathcal A[1]\rangle$, then for any $E\in\mathcal C$ there is an exact triangle
\begin{equation}\label{triangle C A}E'\overset{a}{\to} E\to E''\to E'[1]
\end{equation}
with $E'\in\mathcal C\cap(\mathcal A[1]),E''\in\mathcal C\cap\mathcal A$.
\end{lemma} 

\textit{Proof.} By \cite[Proposition 2.3.2 (b)]{BMT} $(\mathcal A\cap\mathcal C,\mathcal A\cap(\mathcal C[-1]))$ is a torsion pair on the category $\mathcal A$. In particular, $\mathcal A\subset\langle\mathcal C[-1],\mathcal C\rangle$. Applying \cite[Proposition 2.3.2 (b)]{BMT} again, we get that $(\mathcal C[-1]\cap\mathcal A,\mathcal C[-1]\cap(\mathcal A[-1]))$ is torsion pair on the category $\mathcal C[-1]$. Hence $(\mathcal C\cap(\mathcal A[1]),\mathcal C\cap\mathcal A)$ is a torsion pair on the category $\mathcal C$ and for any $E\in\mathcal C$ we have an exact triple (\ref{triangle C A}). $\square$

In \cite{Sch18,Fano} an analogue of the following Proposition was used implicitly. Let us justify it in more detail.

\begin{proposition}\label{tilted heart}Suppose that $\mathcal C$ is the heart of a bounded t-structure on $\mathrm D^b(X)$ satisfying the properties (i) and (ii) of the category $\mathfrak C$ from Lemma \ref{two properties}. For any
	$\gamma\in\mathbb R$ we define a torsion pair
	\begin{equation}\label{T'' F''}
		\begin{split}
			& \mathcal T''_{\gamma}=\{E\in\mathcal A^{\alpha,\beta}(X)\ |\ \text{any quotient}\ E\twoheadrightarrow G
			\,\text{satisfies}\,\lambda_{\alpha,\beta, s}(G)>\gamma \}, \\
			& \mathcal F''_{\gamma} = \{E \in \mathcal A^{\alpha, \beta}(X)\ |\ \text{any
				subobject}\ 0\ne F\hookrightarrow E\,\text{satisfies}\, \lambda_{\alpha,\beta, s}(F)\leq\gamma\}.
		\end{split}
	\end{equation}
	Then there exists a $\gamma\in\mathbb R$ such that
	$$\langle \mathcal T''_{\gamma},\mathcal F''_{\gamma}[1]\rangle =\mathcal C.$$
\end{proposition}

\textit{Proof.} We take $\gamma=-\cot \pi\phi_1$, where $\phi_1$ satisfies the property (i) of Lemma \ref{two properties}. Since both $\mathcal C$ and $\langle \mathcal T''_{\gamma},\mathcal F''_{\gamma}[1]\rangle$ are hearts of bounded t-structures on $\mathrm D^b(X)$, it suffices to prove that $\mathcal C\subset\langle \mathcal T''_{\gamma},\mathcal F''_{\gamma}[1]\rangle$. 
	
Take $E\in\mathcal C$ to show that $E\in\langle \mathcal T''_{\gamma},\mathcal F''_{\gamma}[1]\rangle$. We denote $\mathcal A:=\mathcal A^{\alpha,\beta}(X)$. By Lemma \ref{C A} we reduce the problem to the case when $E\in\mathcal C\cap\mathcal A$ or $E\in\mathcal C\cap(\mathcal A[1])$. Assume that $E\in\mathcal C\cap(\mathcal A[1])$ (for $E\in\mathcal C\cap\mathcal A$ one can give a dual argument). If $E\notin\mathcal F''_\gamma[1]$, then there exists a subobject $f:F\hookrightarrow E[-1]$ in $\mathcal A$ with $\lambda_{\alpha,\beta,s}(F)>\gamma$. Note that, since $0\neq F\in\mathcal A$ and $(\mathcal A,Z_{\alpha,\beta,s})$ is a Bridgeland stability condition, we have $Z_{\alpha,\beta,s}(F)\neq 0$. By Lemma \ref{C A} there is an exact triangle $F'\overset{f'}{\to} F\to F''\to F'[1]$ with $F'\in\mathcal A\cap\mathcal C,F''\in\mathcal A\cap(\mathcal C[-1])$. The object $F'$ should be nonzero by the assumption on the $\lambda_{\alpha,\beta,s}$-slope of $F$ and the property (i) of Lemma \ref{two properties}, hence the morphism $f':F'\to F$ is also nonzero. The composition $F'\overset{f'}{\to}F\overset{f}{\to}E[-1]$ has to be nonzero, since it is the composition of a nonzero morphism in $\mathcal A$ with a monomorphism $f$. However, $F'\in\mathcal C,E[-1]\in\mathcal C[-1]$, and existence of such a morphism contradicts the definition of a t-structure. $\square$

Lemma \ref{two properties} and Proposition \ref{tilted heart} together imply the following corollary.

\begin{corollary}\label{heart}Let $(\beta,\alpha)\in D$, and let $ \mathcal T''_{\gamma}$, $\mathcal F''_{\gamma}$ be the torsion pair on the category $\mathcal A^{\alpha,\beta}(X)$ defined in (\ref{T'' F''}). There are $s>\frac 16$ and $\gamma\in\mathbb R$ such that
$\langle \mathcal T''_{\gamma},\mathcal F''_{\gamma}[1]\rangle =\mathfrak C.$
\end{corollary}

\begin{remark}\label{complexes}
Any object $E\in\mathfrak C$ is isomorphic to a complex of the form
$$\mathcal O_{X}(-1)^{\oplus a}\to\mathcal Q(-1)^{\oplus b}\to\mathcal U^{\oplus c}\to\mathcal O_{X}^{\oplus d},$$
concentrated in degrees $-3,-2,-1$ and $0$. The numbers $a,b,c,d$ are uniquely determined by the Chern character of $E$.
\end{remark}

Also, for further reference we include the following result (the idea is from \cite[Section 4]{EHV}).

\begin{lemma}\label{reflexive}Let $E$ be a reflexive sheaf on a smooth projective threefold $X$. Then $\mathrm{ch}_i(E^\vee)=\mathrm{ch}_i(E)$ for $i\le 2$, and $\mathrm{ch}_3(E)+\mathrm{ch}_3(E^\vee)=h^0(\mathcal Ext^1(E,\mathcal O_X))\ge 0$.
\end{lemma}

\textit{Proof.} The result \cite[Proposition 1.3]{H} of Hartshorne together with Auslander-Buchsbaum formula implies that the projective dimension of $E$ at any point of $X$ is smaller than or equal to 1. Therefore, we can construct a resolution
$$0\to\mathcal E_1\to\mathcal E_0\to E\to 0$$
with locally free sheaves $\mathcal E_i$. Dualizing, we obtain an exact sequence
$$0\to E^\vee\to\mathcal E_0^\vee\to\mathcal E_1^\vee\to\mathcal Ext^1(E,\mathcal O_X)\to 0.$$

By \cite[Corollary 1.4]{H} the sheaf $E$ is locally free on the complement of a closed set of dimension $\le 0$, hence $\dim\mathrm{supp}(\mathcal Ext^1(E,\mathcal O_X))\le 0$. Such a sheaf can be obtained from skyscraper sheaves by extensions, and Grothendieck-Riemann-Roch Theorem implies, that $\mathrm{ch}_i(\mathcal Ext^1(E,\mathcal O_X))=0$ for $i\le 2$, and $\mathrm{ch}_3(\mathcal Ext^1(E,\mathcal O_X))=h^0(\mathcal Ext^1(E,\mathcal O_X))$. Using additivity of Chern character in exact sequences and the property $\mathrm{ch}_i(\mathcal E^\vee)=(-1)^i\mathrm{ch}_i(\mathcal E)$ for a locally free sheaf $\mathcal E$, we obtain the statement of the lemma. $\square$

\section{Objects with Maximal Third Chern Character}
\label{bounds}
We first deal with the twisted trivial sheaves and their shifts.

\begin{proposition}\label{trivial} If $E\in\mathrm{D}^b(X)$ is a tilt-semistable or Bridgeland semistable object for some $\alpha>0,\beta\in\mathbb R$ with $\mathrm{ch}(E)=\mathrm{ch}(\mathcal O_X(n)^{\oplus m})$, then $E\cong\mathcal O_X(n)^{\oplus m}$. Moreover, if $E$ is tilt-semistable and $\mathrm{ch}_{\le 2}(E)=\mathrm{ch}_{\le 2}(\mathcal O_X(n)^{\oplus m})$, then $\mathrm{ch}_3(E)\le\mathrm{ch}_3(\mathcal O_X(n)^{\oplus m})$.
\end{proposition}

\textit{Proof.}
We follow the lines of the proof of \cite[Proposition 4.1]{Sch19}. Tensoring $E$ with $\mathcal O_X(-n)$ we reduce to the case $n=0$. We prove the statement for objects which are Bridgeland semistable for $\alpha=\frac 14,\beta=-\frac 23$ and $s>\frac 16$ from Corollary \ref{heart}. By this Corollary either $E$ or $E[1]$ belongs to $\mathfrak C$. The Chern character assumption implies that $E\in\mathfrak C$ and $E\cong\mathcal O_X^{\oplus m}$. The rest of the proof of the first statement repeats the arguments in \cite[Proposition 4.1]{Sch19} verbatim.

The second statement is can analogously be reduced to the case $n=0$. Then the fact that $E\in\mathrm{Coh}^\beta(X)$ implies that $\beta<0$, and then the generalized Bogomolov--Gieseker inequality (Proposition \ref{BMT}) implies that $\mathrm{ch}_3(E)\le 0$. $\square$

The following statement may be used for dealing with semistable objects of negative rank.

\begin{proposition}[{\cite[Proposition 5.1.3]{BMT}}]\label{duality} For any $\nu_{\alpha,\beta}$-semistable object $E\in\mathrm{Coh}^\beta(X)$ with $\nu_{\alpha,\beta}(E)\neq +\infty$ there exists a $\nu_{\alpha,-\beta}$-semistable object $\widetilde{E}\in\mathrm{Coh}^{-\beta}(X)$ and a sheaf $T$ supported in dimension zero together with a distinguished triangle
\begin{equation}\label{triangle}\widetilde{E}\to\mathbf{R}\mathcal{H}om(E,\mathcal O_X)[1]\to T[-1]\to\widetilde{E}[1].\end{equation}
\end{proposition}

\begin{corollary}\label{-}If $E\in\mathrm{Coh}^\beta(X)$ is an object with $\mathrm{ch}_{\le 2}(E)=-\mathrm{ch}_{\le 2}(\mathcal O_X(n)^{\oplus m})$, which is tilt-semistable for some $\beta>n,\alpha>0$, then $\mathrm{ch}_3(E)\le\mathrm{ch}_3(\mathcal O_X(n)^{\oplus m}[1])$, and in case of equality we have $E\cong\mathcal O_X(n)^{\oplus m}[1]$.
\end{corollary}

\textit{Proof.} Applying Proposition \ref{duality} to $E$ we obtain a tilt-semistable object $\widetilde{E}$ with $\mathrm{ch}_{\le 2}(\widetilde E)=\mathrm{ch}_{\le 2}(\mathcal O_X(-n)^{\oplus m})$, and the exact triangle (\ref{triangle}) gives $\mathrm{ch}_3(E)=\mathrm{ch}_3(\widetilde{E})-\mathrm{ch}_3(T)$. Proposition \ref{trivial} says that $\mathrm{ch}_3(\widetilde{E})\le\mathrm{ch}_3(\mathcal O_X(-n)^{\oplus m})$, which gives the desired bound on $\mathrm{ch}_3(E)$. In case when this bound is achieved we have $T=0$ and $\mathbf{R}\mathcal{H}om(E,\mathcal O_X)[1]\cong\mathcal O_X(-n)^{\oplus m}$. Applying the dualizing functor $\mathbf{R}\mathcal{H}om(-,\mathcal O_X)[1]$ to both sides we get $E\cong\mathcal O_X(n)^{\oplus m}[1]$. $\square$

Now we move to the main result of this article.

\begin{theorem}\label{main}Let $E\in\mathrm{Coh}^\beta(X)$ be a tilt-semistable object with $\mathrm{ch}(E)=2+cH+dH^2+eH^3$.\\
	(1) If $c=-1$, then $d\le \frac{1}{10}$. \\
	(1.1) If $d=\frac{1}{10}$, then $e\le \frac{1}{30}$. In case of equality we have $E\cong\mathcal U$.\\
	(1.2) If $d=-\frac{1}{10}$, then $e\le\frac{7}{30}$. In case of equality $E$ is included into an exact triple
	$$0\to\mathcal O_X(-1)^{\oplus 4}\to E\to \mathcal U(-1)[1]\to 0.$$  
	(2) If $c=0$, then $d\le 0$.\\
	(2.1) If $d=0$, then $e\le 0$. In case of equality $E\cong\mathcal O_X^{\oplus 2}$.\\
	(2.2) If $d=-\frac 15$, then $e\le-\frac{1}{5}$. In case of equality $E$ is a coherent sheaf, included into an exact triple
	$$0\to E\to\mathcal O_X^{\oplus 2}\to\mathcal O_L(1)\to 0$$
	for a projective line $L\subset X$. \\
	(2.3) If $d=-\frac 25$, then $e\le 0$. In case of equality $E$ is included into an exact triple
	$$0\to\mathcal U^{\oplus 4}\to E \to \mathcal Q(-1)^{\oplus 2}[1]\to 0.$$
\end{theorem}   

We consider the above cases in a sequence of lemmas. Firstly, let us prove the following auxiliary result.

\begin{lemma}\label{3 -2 3/5}If an object $E\in\mathrm{Coh}^\beta(X)$ is tilt-semistable and $\mathrm{ch}(E)=3-2H+\frac 35H^2+eH^3,$ then $e\le-\frac{7}{30}$.
\end{lemma}

\textit{Proof.} If $E$ is not tilt-semistable for some $\beta<-\frac 23,\alpha>0$, then there is a semicircular wall $W$, which should intersect the ray $\beta=\beta_-(E)=\frac{-2-\sqrt{2/5}}{3}$ (by Proposition \ref{numerical} (2, 3)). Lemma \ref{stable} implies that this wall is induced by a tilt-stable subobject or quotient $F$ with $\mathrm{ch}_{\le 2}(F)=s+xH+yH^2,s>0,\mu(F)<\mu(E)$. Since $0<\mathrm{ch}_1^{\beta_-(E)}(F)<\mathrm{ch}_1^{\beta_-(E)}(E)$, we get
\begin{equation}\label{x1}0<x+\frac 23s+\frac{\sqrt{2/5}}{3}s<\sqrt{\frac 25}.
\end{equation}
Also we have $-1<\beta_-(E)<\beta_-(F)\le\mu(F)<\mu(E)<0$ (cf. Proposition \ref{actual} (5)) and $\beta_+(F)=2\mu(F)-\beta_-(F)<2\mu(E)-\beta_-(E)<0$. Hence $\beta_-(F),\beta_+(F)\in[-1,0)$. If $s\neq 1$, we can apply Theorem \ref{rank bound} to $F$ and get that $s^2\le\overline{\Delta}_H(F)<\overline{\Delta}_H(E)=10$. Hence, $s\in\{1,2,3\}$. A direct calculation shows that for any such $s$ there are no integer $x$ satisfying (\ref{x1}). Therefore, $E$ is tilt-semistable for all $\beta<-\frac 23,\alpha>0$.

Proposition \ref{2-stability} implies that $E$ should be a 2-semistable sheaf, in particular, $H^0(E)=0$. If $H^2(E)\neq 0$, then by Serre duality there is a nonzero morphism $E\to\mathcal O_X(-2)[1]$. However, such a morphism would destabilize $E$ below the wall $W(E,\mathcal O_X(-2)[1])$, contradicting what we have proved. We have $\mathrm{td}(T_{X})=1+H+\frac{8}{15}H^2+\frac 15H^3$ (see \cite[Lemma 2.1]{Qin}). Grothendieck-Riemann-Roch Theorem gives that $\chi(E)=5e+\frac 23=-h^1(E)-h^3(E)\le 0$, hence $e\le-\frac{2}{15}$.

Since $-1<\beta_-(E)<-\frac 12$, Lemma \ref{nu-lambda} implies that $E[1]$ is $\lambda_{\alpha,\beta,s}$-semistable for some $(\beta,\alpha)\in D$ with infinite $\lambda_{\alpha,\beta,s}(E[1])$. Then Corollary \ref{heart} implies that $E[1]$ belongs to $\mathfrak C$. For the case $e=-\frac{2}{15}$ we find that $E\cong\mathcal O_X(-1)[2]\oplus\mathcal U$. However, this object does not belong to $\mathrm{Coh}^\beta(X)$, so this case is impossible. Since the step by which $\mathrm{ch}_3$ can change for a fixed $\mathrm{ch}_{\le 2}$ is equal to $\frac{H^3}{10}$ (because $c_3$ belongs to $H^6(X,\mathbb Z)$), we get that $e\le-\frac{7}{30}$. $\square$

\begin{lemma}\label{U} If $E\in\mathrm{Coh}^\beta(X)$ is tilt-semistable and $\mathrm{ch}(E)=2-H+\frac{H^2}{10}+eH^3$, then $e\le\frac{1}{30}$, and in case of equality we have $E\cong\mathcal U$.
\end{lemma}

\textit{Proof.} Suppose that $E$ is destabilized at a semicircular wall by a tilt-stable subobject or quotient $F$ with $\mathrm{ch}_{\le 2}(F)=s+xH+yH^2$. The wall should intersect the ray $\beta=\beta_-(E)=-\frac 12-\sqrt{\frac{3}{20}}$. Again we may assume that $s>0,\mu(F)<\mu(E)$. Since $0<\mathrm{ch}_1^{\beta_-(E)}(F)<\mathrm{ch}_1^{\beta_-(E)}(E)$, we have
\begin{equation}\label{x2}0<x+\frac s2+\sqrt{\frac{3}{20}}s<\sqrt{\frac 35}.
	\end{equation}
We also have $-1<\beta_-(E)<\beta_-(F)\le\mu(F)<\mu(E)<0$ and $\beta_+(F)=2\mu(F)-\beta_-(F)<2\mu(E)-\beta_-(E)<0$. If $s\neq 1$, then, applying Theorem \ref{rank bound}, we get that $s^2\le\overline{\Delta}_H(F)<\overline{\Delta}_H(E)=15$. So $s\in\{1,2,3\}$. The only option for an integer $x$ in (\ref{x2}) is $s=3,x=-2$. Using the inequalities $0\le\overline{\Delta}_H(F)<\overline{\Delta}_H(E)$ we get that $y=\frac 35$. Then $\mathrm{ch}_{\le 2}(E/F)=-1+H-\frac{H^2}{2}$ (by abuse of notation we assume that $F$ is a subobject; if it is a quotient, then we can consider the kernel of the surjection $E\twoheadrightarrow F$ instead of $E/F$). Lemma \ref{3 -2 3/5} implies that $\frac{\mathrm{ch}_3(F)}{H^3}\le-\frac{7}{30}$; Proposition \ref{-}, which can be applied, since the wall is not vertical, gives $\frac{\mathrm{ch}_3(E/F)}{H^3}\le\frac 16$. Hence in this case $e\le-\frac{1}{15}$. But $E\cong\mathcal U$ has a bigger third Chern character, hence we get that $E$ should be tilt-stable for all $\beta<-\frac 12,\alpha>0$.

Tilt-stability of $E$ implies that $H^0(E)=0$. If $H^2(E)\neq 0$, then there would exist a non-zero morphism $E\to\mathcal O_X(-2)[1]$, contradicting the fact that $E$ is tilt-stable for all $\beta<-\frac 12,\alpha>0$. Grothendieck-Riemann-Roch Theorem gives that $\chi(E)=5e-\frac 16\le 0$. This gives $e\le\frac{1}{30}$. In case of equality Lemma \ref{nu-lambda} and Corollary \ref{heart} imply that $E[1]$ belongs to $\mathfrak C$. Chern character of $E$ implies that $E\cong\mathcal U$. $\square$

\begin{lemma}\label{2 0 -1/5} If $E\in\mathrm{Coh}^\beta(X)$ is tilt-semistable and $\mathrm{ch}(E)=2-\frac{H^2}{5}+eH^3$, then $e\le-\frac{1}{5}$, and in case of equality $E$ is a coherent sheaf, included into an exact triple
\begin{equation}0\to E\to\mathcal O_X^{\oplus 2}\to\mathcal O_L(1)\to 0\end{equation}
for a projective line $L\subset X$.
\end{lemma}

\textit{Proof.} Any semicircular wall for $E$ has to intersect the ray $\beta=\beta_-(E)=-\frac{\sqrt{5}}{5}$. We can assume that this wall is given by a tilt-stable subobject or quotient $F$ with $\mathrm{ch}_{\le 2}(F)=s+xH+yH^2$, $s>0,\mu(F)<\mu(E)$. Since $0<\mathrm{ch}_1^{\beta_-(E)}(F)<\mathrm{ch}_1^{\beta_-(E)}(E)$, we have
\begin{equation}\label{x3}0<x+\frac{\sqrt{5}}{5}s<\frac{2\sqrt{5}}{5}.
\end{equation}
We also have $-1<\beta_-(E)<\beta_-(F)\le\mu(F)<\mu(E)<0$. Dividing (\ref{x3}) by $s$, we obtain that $\mu(F)<\frac{2\sqrt{5}-s\sqrt{5}}{5s}$. Hence for $s\ge 4$ $\beta_+(F)=2\mu(F)-\beta_-(F)<\frac{4\sqrt{5}-s\sqrt{5}}{5s}\le 0$. In this case Theorem \ref{rank bound} can be applied to $F$ and gives $s^2\le\overline{\Delta}_H(F)<\overline{\Delta}_H(E)=10$. So, we get that $s\in\{1,2,3\}$.

If $s=1$, then (\ref{x3}) implies that $-\frac{\sqrt 5}{5}<x<\frac{\sqrt 5}{5}$, so $x=0$, but in this case the property $\mu(F)<\mu(E)$ does not hold and the wall is not semicircular. If $s=2$, then by (\ref{x3}) $-1<x<0$, which is impossible. And if $s=3$, then (\ref{x3}) implies that $-2<x<0$, so $x=-1$. Inequalities $0\le\overline{\Delta}_H(F)<20$ together with the fact that $y\in-\frac 12+\frac 15\mathbb Z$ (since $c_i(F)\in H^{2i}(X,\mathbb Z)$) imply that $y=\frac{1}{10}$. But then we find that $\overline{\Delta}_H(F)=\overline{\Delta}_H(E/F)=10$, contradicting the inequality $\overline{\Delta}_H(F)+\overline{\Delta}_H(E/F)<\overline{\Delta}_H(E)$. Therefore, $E$ is $\nu_{\alpha,\beta}$-semistable for all $\alpha>0,\beta<0$.

Tilt-semistability of $E$ implies that $H^0(E)=0$. Also $H^2(E)\cong\mathrm{Hom}(E,\mathcal O_X(-2)[1])^\vee=0$, because otherwise there would be a semicircular wall. Grothendieck-Riemann-Roch Theorem gives $\chi(E)=5e+1\le 0$, hence $e\le-\frac 15$. Assume that $e=-\frac 15$. Proposition \ref{2-stability} implies that $E$ is a 2-semistable sheaf. Since $E$ is a rank 2 sheaf with negative third Chern class, it cannot be reflexive (see the proof of \cite[Proposition 2.6]{H}). The sheaf $E$ can be included into an exact sequence
$$0\to E\to E^{\vee\vee}\to T\to 0,$$
where $\dim \mathrm{supp}(T)\le 1$. If $\dim \mathrm{supp}(T)=0$, then $\nu_{\alpha,\beta}(E^{\vee\vee})=\nu_{\alpha, \beta}(E)$ for all $(\beta,\alpha)$, and tilt-semistability of $E$ implies tilt-semistability of $E^{\vee\vee}$. Applying to $E^{\vee\vee}$ the above arguments, we obtain that $\frac{\mathrm{ch}_3(E^{\vee\vee})}{H^3}\le-\frac 15$, contradicting reflexivity of $E^{\vee\vee}$. 

So, $\mathrm{supp}(T)$ contains 1-dimensional components, and $\frac{H\cdot\mathrm{ch}_2(T)}{H^3}\ge\frac 15$. As Schmidt noted in \cite[Lemma 2.1]{Sch23}, the dual of a slope semistable sheaf is slope semistable, hence $E^{\vee\vee}$ is slope semistable. If it is slope stable, then Bogomolov-Gieseker inequality implies that $\mathrm{ch}_{\le 2}(E)=(2,0,0)$. The generalized Bogomolov-Gieseker inequality implies $\frac{\mathrm{ch}_3(E^{\vee\vee})}{H^3}\le 0$. But, as we already mentioned, the third Chern class of a rank 2 reflexive sheaf has to be non-negative, implying that $\mathrm{ch}_3(E^{\vee\vee})=0$. Proposition \ref{trivial} implies that $E^{\vee\vee}\cong\mathcal O_X^{\oplus 2}$, but this sheaf is not slope stable.

Therefore, $E^{\vee\vee}$ is strictly slope semistable, implying that there is a short exact sequence in $\mathrm{Coh}(X)$
\begin{equation}\label{destab}0\to F\to E^{\vee\vee}\to G\to 0,
\end{equation}
where $F$ and $G$ are torsion free rank one sheaves with $\mu(F)=\mu(G)=0$. Bogomolov-Gieseker inequality implies that $\frac{H\cdot\mathrm{ch}_2(F)}{H^3}$ and $\frac{H\cdot\mathrm{ch}_2(G)}{H^3}$ are non-positive, and positivity of $\frac{H\cdot\mathrm{ch}_2(T)}{H^3}$ implies that $\mathrm{ch}_{\le 2}(F)=\mathrm{ch}_{\le 2}(G)=(1,0,0)$. Generalized Bogomolov-Gieseker inequality implies that $\frac{\mathrm{ch}_3(F)}{H^3}$ and $\frac{\mathrm{ch}_3(G)}{H^3}$ are non-positive, and non-negativity of $c_3(E^{\vee\vee})$ implies their vanishing. By Proposition \ref{trivial} we get $F\cong\mathcal O_X,G\cong\mathcal O_X$, hence $E^{\vee\vee}\cong\mathcal O_X^{\oplus 2}$.

Since $\mathrm{ch}(T)=\frac{H^2}{5}+\frac{H^3}{5}$, one-dimensional part of $\mathrm{supp}(T)$ consists of one projective line $L\subset X$. Classification of coherent sheaves on $\mathbb P^1$ implies that $T\cong\mathcal O_L(a)\oplus T',\dim \mathrm{supp}(T')\le 0$. Chern character of $T$ together with existence of a surjective morphism $\mathcal O_X^{\oplus 2}\twoheadrightarrow T$ implies that $T\cong\mathcal O_L(1)$ or $T\cong\mathcal O_L\oplus\mathcal O_p,p\in X$. However, in the second case one can, using that $H^0(E)=0$, construct a subsheaf $\mathcal I_p\hookrightarrow E$, contradicting 2-stability of $E$. And in case when $T\cong\mathcal O_L(1)$ one can show that $\mathrm{Hom}(\mathcal I_Z,E)=0$ if $\dim Z\le 0$, implying that $E$ is 2-semistable. $\square$

\begin{lemma}\label{2 -1 -1/10} If $E\in\mathrm{Coh}^\beta(X)$ is tilt-semistable and $\mathrm{ch}(E)=2-H-\frac{H^2}{10}+eH^3$, then $e\le\frac{7}{30}$, and in case of equality $E$ is included into an exact triple
\begin{equation}\label{triple}0\to\mathcal O_X(-1)^{\oplus 4}\to E\to \mathcal U(-1)[1]\to 0.\end{equation}
\end{lemma}

\textit{Proof.} Note that $H^2\cdot\mathrm{ch}^{-1}_1(E)$ has the least positive value, so there is no wall for $E$ at $\beta=-1$ for any $\alpha>0$. The inequality $Q_{\alpha,\beta}(E)<0$ defines a semi-disk lying below the numerical wall $Q_{\alpha,\beta}(E)=0$. Therefore the generalized Bogomolov-Gieseker inequality together with the fact that the walls are nested implies that $Q_{0,-1}(E)=41-150e\ge 0$. Using also the integrality of Chern classes of $E$, we obtain the bound $e\le\frac{7}{30}$. Assume that $e=\frac{7}{30}$.
	
A calculation shows that the numerical wall $W(E,\mathcal O_X(-3)[1])$ intersects the ray $\beta=-1$. Therefore, $\mathrm{Ext}^2(\mathcal O_X(-1),E)\cong\mathrm{Hom}(E,\mathcal O_X(-3)[1])=0$. Grothendieck-Riemann-Roch Theorem gives that $\mathrm{Hom}(\mathcal O_X(-1),E)\ge\chi(\mathcal O_X(-1),E)=4$. This gives an exact sequence
\begin{equation}\label{G}0\to\mathcal O_X(-1)^{\oplus 4}\to E\to G\to 0\end{equation}
in the category $\mathrm{Coh}^\beta(X)$ for $\beta<-1$ large enough. Here $\mathrm{ch}(G)=-2+3H-\frac{21}{10}H^2+\frac{9}{10}H^3$. We find $\mathrm{ch}(\mathbf{R}\mathcal{H}om(G,\mathcal O_X)[1](-2))=2-H+\frac{H^2}{10}+\frac{H^3}{30}$. Proposition \ref{duality} gives an exact triangle
$$\widetilde{G}\to \mathbf{R}\mathcal{H}om(G,\mathcal O_X)[1](-2)\to T[-1]\to \widetilde{G}[1]$$
with a tilt-semistable object $\widetilde{G}$ and a sheaf $T$ with $\dim T\le 0$. It follows that $\frac{\mathrm{ch}_3(\widetilde G)}{H^3}=\frac{1}{30}+\frac{\mathrm{ch}_3(T)}{H^3}.$ Since $\frac{\mathrm{ch}_3(T)}{H^3}\ge 0$ and, by Lemma \ref{U}, $\frac{\mathrm{ch}_3(\widetilde G)}{H^3}\le\frac{1}{30}$, we obtain that $T=0$ and $\mathbf{R}\mathcal{H}om(G,\mathcal O_X)[1](-2)\cong\mathcal U$. It follows that $G\cong\mathcal U(-1)[1]$. $\square$

\begin{remark}
Note that the exact sequence (\ref{triple}) gives an exact triangle $$\mathcal O_X(-1)^{\oplus 4}\to E\to \mathcal U(-1)[1]\to \mathcal O_X(-1)^{\oplus 4}[1].$$
Rotating it, we obtain an exact triangle
$$\mathcal U(-1)\to\mathcal O_X(-1)^{\oplus 4}\to E\to \mathcal U(-1)[1].$$
Hence, if $E$ is a sheaf, we have an exact sequence
$$0\to\mathcal U(-1)\to\mathcal O_X(-1)^{\oplus 4}\to E\to 0.$$
\end{remark}

In order to continue, we need several preparatory statements.

\begin{lemma}\label{3 -1 1/10}If $E\in\mathrm{Coh}^\beta(X)$ is tilt-semistable and $\mathrm{ch}(E)=3-H+\frac{1}{10}H^2+eH^3$, then $e\le -\frac 16$.
\end{lemma}

\textit{Proof.} Any semicircular wall for $E$ has to intersect the ray $\beta=\beta_-(E)=\frac{-10-2\sqrt{10}}{30}$. Assume that this wall is given by a tilt-stable subobject or quotient $F$ with $\mathrm{ch}_{\le 2}(F)=s+xH+yH^2, s>0,\mu(F)<\mu(E)$. Since $0<\mathrm{ch}_1^{\beta_-(E)}(F)<\mathrm{ch}_1^{\beta_-(E)}(E)$, we have
	\begin{equation}\label{x4}0<x+\frac 13s+\frac{2\sqrt{10}}{30}s<\frac{2\sqrt{10}}{10}.
	\end{equation}
We have $-1<\beta_-(E)<\beta_-(F)\le\mu(F)<\mu(E)<0$. Also, $\beta_+(F)=2\mu(F)-\beta_-(F)<2\mu(E)-\beta_-(E)=-\frac 23+\frac{10+2\sqrt{10}}{30}<0$. Hence Theorem \ref{rank bound} can be applied to $F$ and gives $s^2\le\overline{\Delta}_H(F)<\overline{\Delta}_H(E)=10$. So, we have $s\in\{1,2,3\}$.

If $s=1$, then the only integer $x$ satisfying (\ref{x4}) is $x=0$. However, in this case the property $\mu(F)<\mu(E)$ does not hold. If $s=2$, then $x=-1$. However, in this case $y\in-\frac 12+\frac 15\mathbb Z$, and the property $\overline{\Delta}_H(F)=25-100y\in [0,10)$ is not satisfied for any such $y$. And if $s=3$, then there are no integer $x$ satisfying (\ref{x4}).

We proved that there are no semicircular walls for $E$. 2-stability of $E$ implies that $H^0(E)=0$, and the absense of semicircular walls implies $H^2(E)=\mathrm{Hom}(E,\mathcal O_X(-2)[1])^\vee=0$. A calculation gives $\chi(E)=5e+\frac 56\le 0$, hence $e\le-\frac 16$. $\square$

\begin{lemma}\label{3 1 1/10}If $E\in\mathrm{Coh}^\beta(X)$ is tilt-semistable and $\mathrm{ch}(E)=3+H+\frac{1}{10}H^2+eH^3$, then $e\le -\frac{2}{15}$.
\end{lemma}

\textit{Proof.} We have $\mathrm{ch}(E(-1))=3-2H+\frac 35H^2+(e-\frac{1}{10})H^3$, and $E(-1)$ is also tilt-semistable. So Lemma \ref{3 -2 3/5} gives $e\le -\frac{2}{15}$. $\square$

\begin{corollary}\label{no2} There are no tilt-semistable objects $E\in\mathrm{Coh}^\beta(X)$ with $\mathrm{ch}_{\le 2}(E)=3-H+\frac{1}{10}H^2$.
\end{corollary}

\textit{Proof.} By the proof of Lemma \ref{3 -1 1/10} such an object $E$ should be tilt-semistable for all $\beta<-\frac 13,\alpha>0$, hence $E$ is a 2-semistable sheaf. The dual sheaf $E^\vee$ should be slope semistable, and in fact slope stable since its rank and first Chern class are coprime. If $E$ is reflexive, then $\mathrm{ch}_{\le 2}(E^\vee)=3+H+\frac{1}{10}H^2$. However, in this case Lemmas \ref{3 -1 1/10} and \ref{3 1 1/10} imply that $\mathrm{ch}_3(E)+\mathrm{ch}_3(E^\vee)<0$, contradicting Lemma \ref{reflexive}, so the sheaf $E$ is not reflexive. There is a short exact sequence
$$0\to E\to E^{\vee\vee}\to T\to 0$$
with $\dim\mathrm{supp}\ T\le 1$. The support of $T$ cannot be zero-dimensional, because in this case $\mathrm{ch}_{\le 2}(E^{\vee\vee})=3-H+\frac{1}{10}H^2$ and by the previous argument $E^{\vee\vee}$ cannot be reflexive. And if $\dim\mathrm{supp}\ T=1$, then $\mathrm{ch}_{\le 2}(E^{\vee\vee})=3-H+(\frac{1}{10}+d)H^2$ for $d\ge\frac 15$ and Bogomolov-Gieseker inequality for $E^{\vee\vee}$ is not satisfied. $\square$

Now we are ready to make the next step.

\begin{lemma}\label{2 0 -2/5} If $E\in\mathrm{Coh}^\beta(X)$ is tilt-semistable and $\mathrm{ch}(E)=2-\frac 25H^2+eH^3$, then $e\le 0$, and in case of equality $E$ is included into an exact triple
	\begin{equation}\label{triple2}0\to\mathcal U^{\oplus 4}\to E\to\mathcal Q(-1)^{\oplus 2}[1]\to 0.\end{equation}
\end{lemma}

\textit{Proof.} Note that $\beta_-(E)=-\frac{\sqrt{10}}{5}\in(-1,-\frac 12)$. In order to use the heart $\mathfrak C$, we need to find a point $(\beta_0,\alpha_0)$ on the hyperbola $\nu_{\alpha,\beta}(E)=-\frac 25+\beta^2-\alpha^2=0$ with $\alpha_0<\beta_0+1$, such that $E$ is $\nu_{{\alpha_0},{\beta_0}}$-semistable. The curve $-\frac 25+\beta^2-\alpha^2=0$ intersects the line $\alpha=\beta+1$ at the point with $\beta=-\frac{7}{10},\alpha=\frac{3}{10}$, hence it suffices to find such a point $(\beta_0,\alpha_0)$ with $\alpha_0<\frac{3}{10}$.
	
Any semicircular wall for $E$ has to be defined by a tilt-stable subobject or quotient $F$ with $\mathrm{ch}_{\le 2}(F)=s+xH+yH^2,s>0,\mu(F)<\mu(E)$. If $s\ge 4$, then Proposition \ref{actual} (1) gives $\rho(E,F)\le\frac{\sqrt{5}}{10}<\frac{3}{10}$. If $E$ is tilt-semistable below some wall of radius $\rho\le\frac{\sqrt{5}}{10}$, then we can apply Lemma \ref{nu-lambda} and Corollary \ref{heart} to get that $E[1]\in\mathfrak C$. And if $E$ is tilt-semistable above all such walls and there are no walls of radius bigger than $\frac{\sqrt{5}}{10}$, then $E$ is, for example, tilt-stable at $(\beta_0,\alpha_0)=(-\sqrt{\frac{37}{80}},\frac 14)$, so $E[1]\in\mathfrak C$ too. So we need to rule out the cases $1\le s\le 3$.

We have $0<\mathrm{ch}_1^{\beta_-(E)}(F)<\mathrm{ch}_1^{\beta_-(E)}(E)$, that is,
	\begin{equation}\label{x5}0<x+\frac{\sqrt{10}}{5}s<\frac{2\sqrt{10}}{5}.
	\end{equation}

If $s=1$, then the only integer $x$ satisfying (\ref{x5}) is $x=0$, but in this case the property $\mu(F)<\mu(E)$ does not hold.

If $s=2$, then from (\ref{x5}) it follows that $x=-1$. We should have $\overline{\Delta}_H(F)=25-100y\in[0,40)$ and $y\in-\frac 12+\frac 15\mathbb Z$, hence $y\in\{-\frac{1}{10},\frac{1}{10}\}$. In the first case $\overline{\Delta}_H(F)=35$, in the second case $\overline{\Delta}_H(F)=15$. However, in both cases $\overline{\Delta}_H(E/F)=25$ and the property $\overline{\Delta}_H(F)+\overline{\Delta}_H(E/F)<\overline{\Delta}_H(E)$ is not satisfied.

Finally, if $s=3$, then (\ref{x5}) implies that $x=-1$. We have $\overline{\Delta}_H(F)=25-150y\in[0,40)$ and $y\in-\frac 12+\frac 15\mathbb Z$, hence $y=\frac{1}{10}$. However, in Corollary \ref{no2} we proved that there are no tilt-stable objects $F$ with such a Chern character. So, we have proved that $E[1]\in\mathfrak C$.

Tilt-semistability of $E$ implies that $H^0(E)=0$. If $H^2(E)\neq 0$, then there is a nonzero morphism $E\to\mathcal O_X(-2)[1]$. However, such a morphism would destabilize $E$ everywhere below the wall $W(E,\mathcal O_X(-2)[1])$, which is located above $W(E,\mathcal O_X(-1)[1])$ and has radius bigger than $\frac{3}{10}$, contradicting what we have proved. A calculation by Grothendieck-Riemann-Roch Theorem gives $\chi(E)=5e\le 0$, hence $e\le 0$. Since for $e=0$ $\mathrm{ch}(E[1])=4\mathrm{ch}(\mathcal U[1])+2\mathrm{ch}(\mathcal Q(-1)[2])$, we get 
an exact sequence of the form (\ref{triple2}) (using Remark \ref{complexes}). $\square$

\begin{remark}\label{Serre}Using the Serre correspondence, one can prove that there exist slope stable rank 2 vector bundles $E$ on $X$ with $c_1(E)=0,c_2(E)=2$ and $H^1(E(-1))=0$, the so-called instanton bundles \cite[Theorem D, Step 1 of the proof]{Even and odd}. It follows that there exist tilt-stable sheaves $E$ from (\ref{triple2}).
\end{remark}

\textit{Proof of Theorem \ref{main}.} The bounds on $\mathrm{ch}_2(E)$ follow from Bogomolov-Gieseker Inequality. Other statements are proven in Proposition \ref{trivial} and Lemmas \ref{U}, \ref{2 0 -1/5}, \ref{2 -1 -1/10}, \ref{2 0 -2/5}. $\square$

\section{Partial Results and a Conjecture}\label{conjectures}

\begin{lemma}\label{bound0} Let $E\in\mathrm{Coh}^\beta(X)$ be a tilt-semistable object with $\mathrm{ch}(E)=H+dH^2+eH^3$. Then $e\leq\frac{1}{24}+\frac{d^2}{2}.$
\end{lemma}

\textit{Proof.} The proof of \cite[Lemma 3.3]{MS18} holds true for any smooth projective threefold over $\mathbb C$, for which the generalized Bogomolov-Gieseker inequality (Proposition \ref{BMT}) holds, in particular, for $X_5$. $\square$

\begin{lemma} Let $E\in\mathrm{Coh}^\beta(X)$ be a tilt-semistable object with $\mathrm{ch}(E)=1-dH^2+eH^3$, then $e\leq\frac{d(d+1)}{2}$. \end{lemma}

\textit{Proof.} Note that $H^2\cdot\mathrm{ch}^{-1}_1(E)$ has the least positive value, so $Q_{0,-1}(E)\geq 0$. This, together with the fact that $e\in\frac{1}{10}\mathbb Z$, implies that for $d=0$ or $d=\frac 15$ we have $e\leq 0$; for $d=\frac 25$ we have $e\leq\frac 15$; for $d=\frac 35$ we have $e\leq\frac 25$ and for $d=\frac 45$ we have $e\leq\frac 35$. For $d\geq 1$ the statement follows from arguments in the proof of \cite[Proposition 3.2]{MS18}. $\square$

Denote by $S$ a smooth hyperplane section of $X\subset\mathbb P^6$. It is known that $S$ is a del Pezzo surface of degree 5, isomorphic to a blow up of $\mathbb P^2$ in 4 generic points. Denote by $p:S\to\mathbb P^2$ the projection of the blow up, denote by $H'$ the class of $p^*(H)$ in $\mathrm{Cl}\ S$ and by $E_1,E_2,E_3,E_4$ the exceptional divisors corresponding to the blown up points. We can compute that $c_1(T_S)=-c_1(\omega_S)=3H'-\sum_{i=1}^4E_i$ and $c_2(T_S)=c_2(T_{\mathbb P^2})+4=7$ (using \cite[Example 15.4.3]{Fulton}), which implies that $\mathrm{td}(T_S)=1+\frac{3H'-\sum_{i=1}^4E_i}{2}+P$ (here $P$ is the class of a point). Denote by $D$ the divisor $aH'+\sum_{i=1}^4b_iE_i$ and by $i:S\to X$ the canonical injection.

\begin{lemma}\label{GRR}We have 
$$\mathrm{ch}(i_*\mathcal O_S(D))=H+\left(\frac{3a+\sum\limits_{i=1}^4b_i}{5}-\frac 12\right)H^2+\left(\frac{a^2-3a-\sum\limits_{i=1}^4b_i^2-\sum\limits_{i=1}^4b_i}{10}+\frac 16\right)H^3.$$\end{lemma}

\textit{Proof.} Note that the embedding $S\to\mathbb P^5$ is given by the anticanonical linear system $|3H'-\sum_{i=1}^4E_i|$ (indeed, adjunction formula implies that $\omega_S^\vee\cong\mathcal O_X(1)|_S$). It follows that $i_*(H')=(3H'-\sum_{i=1}^4E_i,H')L=\frac{3H^2}{5}$ and $i_*(E_j)=(3H'-\sum_{i=1}^4E_i,E_j)L=\frac{H^2}{5}$ for all $j$. Now a straightforward calculation using the Grothendieck--Riemann--Roch Theorem gives the statement of the lemma. $\square$

\begin{proposition}\label{O_S(D)}Denote $3a+\sum_{i=1}^4b_i$ by $C$. In the following $k\in\mathbb Z$.
	\begin{enumerate}
		\item If $C=5k$, then $\frac{\mathrm{ch}_3(i_*\mathcal O_S(D))}{H^3}\le\frac{k^2-k}{2}+\frac 16$. In case of equality $a=3k,$ all $b_i$ are equal to $-k$.
		\item If $C=5k+1$, then $\frac{\mathrm{ch}_3(i_*\mathcal O_S(D))}{H^3}\le\frac{5k^2-3k-2}{10}+\frac 16$. In case of equality one of the following holds:
		\begin{enumerate}
			\item $a=3k$, one of $b_i$ is equal to $1-k$, while other three are equal to $-k$;
			\item $a=3k+1$, two of $b_i$ are equal to $-k$, while other two are equal to $-k-1$.
		\end{enumerate}
	    \item If $C=5k-1$, then $\frac{\mathrm{ch}_3(i_*\mathcal O_S(D))}{H^3}\le\frac{5k^2-7k}{10}+\frac 16$. In case of equality one of the following holds:
	    \begin{enumerate}
	    	\item $a=3k$, one of $b_i$ is equal to $-1-k$, while other three are equal to $-k$;
	    	\item $a=3k-1$, two of $b_i$ are equal to $-k$, while other two are equal to $1-k$.
	    \end{enumerate}
        \item If $C=5k+2$, then $\frac{\mathrm{ch}_3(i_*\mathcal O_S(D))}{H^3}\le\frac{5k^2-k-2}{10}+\frac 16$. In case of equality one of the following holds:
        \begin{enumerate}
        	\item $a=3k+1$, one of $b_i$ is equal to $-1-k$, while other three are equal to $-k$;
        	\item $a=3k+2$, all $b_i$ are equal to $-1-k$.
        \end{enumerate}
        \item If $C=5k-2$, then $\frac{\mathrm{ch}_3(i_*\mathcal O_S(D))}{H^3}\le\frac{5k^2-9k+2}{10}+\frac 16$. In case of equality one of the following holds:
        \begin{enumerate}
        	\item $a=3k-1$, one of $b_i$ is equal to $1-k$, while other three are equal to $-k$;
        	\item $a=3k-2$, all $b_i$ are equal to $1-k$.
        \end{enumerate}
	\end{enumerate}
\end{proposition}

\textit{Proof.}
Lemma \ref{GRR} implies that for a fixed $C$ maximizing $\mathrm{ch}_3(i_*\mathcal O_S(D))$ is equivalent to maximizing $Q:=a^2-\sum_{i=1}^4b_i^2$. For a fixed $C=3a+\sum_{i=1}^4b_i$, a fixed $a\in\mathbb R$ and varying $(b_i)\in\mathbb R^4$ the maximal value of $Q$ is attained when all $b_i$ are equal to $\frac{C-3a}{4}$. The values of $b_i$ in Proposition \ref{O_S(D)} are obtained by choosing a point $\textbf b=(b_i)\in\mathbb Z^4$ with $\sum_{i=1}^4b_i=C-3a$ that is closest to the point with all $b_i$ equal to $\frac{C-3a}{4}$ in the usual Euclidean metric. By Pythagoras' Theorem this $\textbf b$ is indeed a point minimizing $\sum_{i=1}^4b_i^2$ and maximizing $Q$ for fixed $C$ and $a$. It remains to prove that maximum of $Q$ can be reached only for the values of $a$ given in the Proposition.

For fixed $C$ and $a$ we have $Q\le f_{C,a}:=a^2-4\left(\frac{C-3a}{4}\right)^2=-\frac 54a^2+\frac 32Ca-\frac{C^2}{4}$. This is a parabola in $a$ with a maximum at $a_0=\frac{3C}{5}$. This proves the Proposition in the case $C=5k,k\in\mathbb Z$. In the case $C=5k+1,k\in\mathbb Z$ $a_0=\frac{3C}{5}$ lies in the interval $(3k,3k+1)$. For the values of $a,b_i$ from Proposition \ref{O_S(D)} (2) we have $Q=5k^2+2k-1$. A calculation gives $f_{5k+1,3k-1}=5k^2+2k-3$. Hence for all $a\le 3k-1$ $Q\le f_{5k+1,a}\le f_{5k+1,3k-1}<5k^2+2k-1$. Also we find $f_{5k+1,3k+2}=5k^2+2k-\frac 94$, hence for all $a\ge 3k+2$ $Q\le 5k^2+2k-\frac 94<5k^2+2k-1$. Hence $a\in\{3k,3k+1\}$. This proves the case (2). The cases (3-5) can be proven completely analogously. $\square$

The cases of rank two semistable sheaves with maximal $\mathrm{ch}_3$ on $\mathbb P^3$ \cite{Sch18} and on a smooth quadric threefold \cite{Fano} lead us to the following conjecture.

\begin{conjecture}\label{conj2}Suppose that $E\in\mathrm{Coh}(X)$ is a Gieseker semistable sheaf with $\mathrm{ch}(E)=2+cH+dH^2+eH^3$.\begin{enumerate}
\item Let $c=-1$ and $d=\frac 12+\frac C5\ll 0$, $C\in\mathbb Z$. Then $e\le-\frac 13+\frac{\mathrm{ch}_3(i_*\mathcal O_S(D))}{H^3}$, where $D$ is a divisor on $S$ from Proposition \ref{O_S(D)} for the given $C=3a+\sum_{i=1}^4b_i$. A general sheaf $E$ with maximal $e$ can be included into an exact sequence
\begin{equation}\label{E1}0\to\mathcal O_X(-1)^{\oplus 2}\to E\to\mathcal O_S(D)\to 0.\end{equation}
\item Let $c=0$ and $d=-\frac 25+\frac C5\ll 0$, $C\in\mathbb Z$. Then $e\le\frac {1}{30}+\frac{\mathrm{ch}_3(i_*\mathcal O_S(D))}{H^3}$, where $D$ is a divisor on $S$ from Proposition \ref{O_S(D)} for the given $C=3a+\sum_{i=1}^4b_i$. A general sheaf $E$ with maximal $e$ can be included into an exact sequence
\begin{equation}\label{E2}0\to\mathcal U\to E\to\mathcal O_S(D)\to 0.\end{equation}
\end{enumerate}
\end{conjecture}

Note that, if $C=5k,k\in\mathbb Z$, then $i_*\mathcal O_S(D)\cong (i_*\mathcal O_S)(k)$ is a twisted structure sheaf of $S\subset X$. In this case the moduli spaces of stable sheaves $E$ from (\ref{E1}) and (\ref{E2}) were described in \cite[Theorem 4.2]{Fano} and \cite[Theorem 4.6]{Fano}, respectively. In particular, it was proven that moduli spaces of these sheaves form smooth dense open subsets of rational irreducible components of the moduli scheme of Gieseker semistable sheaves on $X$. Note that in (\ref{E1}) and (\ref{E2}) $S$ denotes a smooth hyperplane section of $X\subset\mathbb P^6$, while in \cite{Fano} we considered a more general case in which $S$ was an arbitrary effective divisor.

Let us give some further evidence for Conjecture \ref{conj2}. Note that for an extension
$$0\to\mathcal O_X(-1)^{\oplus 2}\to E\to G\to 0$$
with $\mathrm{ch}(G)=H+\left(k-\frac{9}{10}\right)H^2+\left(\frac{5k^2-9k+2}{10}+\frac 16\right)H^3,k\in\mathbb Z$ we have $\mathrm{ch}(E)=2-H+\left(k+\frac{1}{10}\right)H^2+\frac{15k^2-27k+1}{30}H^3$.

\begin{lemma}\label{our bound} Let $E\in\mathrm{Coh}^\beta(X)$ be a tilt-semistable object with $\mathrm{ch}(E)=2-H+\left(k+\frac{1}{10}\right)H^2+eH^3$. If $k\le -2$ and $e\ge \frac{15k^2-27k+1}{30}$, then $E$ is destabilized along a semicircular wall with the subobject and the quotient of rank at most two.\end{lemma}

\textit{Proof.} We can calculate that the radius $\rho_Q$ of the numerical wall $Q_{\alpha,\beta}(E)=0$ satisfies the inequality
$$\frac{\rho_Q}{\overline{\Delta}_H(E)}\ge-\frac{9}{800}k-\frac{11}{2000}-\frac{147}{8000-20000k}+\frac{2401}{-2\cdot 10^6k^3+24\cdot 10^5k^2-96\cdot 10^4k+128000}.$$

Bounding each term, we see that $\frac{\rho_Q}{\overline{\Delta}_H(E)}>\frac{1}{300}$, hence we can apply Proposition \ref{actual} (1) to get the statement of our lemma. $\square$

\begin{lemma}\label{rank one} Let $E\in\mathrm{Coh}^\beta(X)$ be a tilt-semistable object with $\mathrm{ch}(E)=2-H+\left(k+\frac{1}{10}\right)H^2+eH^3$, $k\in\mathbb Z$. If $k\le -2$ and $E$ is destabilized by a subobject $F$ of rank one, then $e\le \frac{15k^2-27k+1}{30}$.
\end{lemma}

\textit{Proof.} Follows along the same lines as the proof of \cite[Lemma 3.9]{Sch18}. Note that in our case
$$\frac 12\left(\frac{H\cdot\mathrm{ch}_2(E)}{H^3}\right)^2+\frac{H\cdot\mathrm{ch}_2(E)}{H^3}+\frac{5}{24}=\frac{15k^2-27k+1}{30}+\frac{2}{25}.$$
Since the step by which $\mathrm{ch}_3(E)$ can change for a fixed $\mathrm{ch}_{\le 2}(E)$ is equal to $\frac{H^3}{10}$, the bound proven in \cite[Lemma 3.9]{Sch18} suffices to prove our lemma. $\square$

\vspace{5mm}

\noindent
Danil A.~Vassiliev\\
National Research University Higher School of Economics, Russian Federation\\
AG Laboratory, HSE, 6 Usacheva str., Moscow, Russia, 119048\\
\textit{E-mail}:{\ danneks@yandex.ru}

\end{document}